\documentclass[11pt]{amsart}

\usepackage{amssymb,amscd,amsmath,color,enumerate,mathabx,mathtools}
\usepackage[all,cmtip]{xy}
\usepackage{mathrsfs}
\usepackage{lipsum}
\usepackage{nicematrix,tikz, ifthen}
\usepackage{scalerel}
\usepackage{ascii}

\usepackage{stmaryrd}

\usepackage[T1]{fontenc}

\usepackage[colorlinks,
linkcolor=black!75!red,
citecolor=blue,
pdftitle={},
pdfproducer={pdfLaTeX},
bookmarksopen=true,
bookmarksnumbered=true,
backref=page]{hyperref}



\usepackage{extarrows}

\newcommand{\xrightarrowdbl}[2][]{%
  \xrightarrow[#1]{#2}\mathrel{\mkern-14mu}\rightarrow
}

\usepackage[shortlabels]{enumitem}
\usepackage{tensor}
\usepackage{braket}

\usepackage[margin=3cm]{geometry}

\newcommand{\R}{{\ensuremath{\mathbb{R}}}}

\newcommand{\Z}{{\ensuremath{\mathbb{Z}}}}

\newcommand{\C}{{\ensuremath{\mathbb{C}}}}

\newcommand\bC{{\mathbb C}}

\newcommand\bG{{\mathbb G}}

\newcommand\bP{{\mathbb P}}

\newcommand\bR{{\mathbb R}}
\newcommand\bS{{\mathbb S}}
\newcommand\bT{{\mathbb T}}

\newcommand\bZ{{\mathbb Z}}

\newcommand\cC{{\mathcal C}}

\newcommand\cP{{\mathcal P}}

\newcommand\cU{{\mathcal U}}

\newcommand\cX{{\mathcal X}}

\usepackage[normalem]{ulem}

\newcommand{\stkout}[1]{\ifmmode\text{\sout{\ensuremath{#1}}}\else\sout{#1}\fi}

\DeclareMathOperator{\GL}{GL}
\DeclareMathOperator{\SL}{SL}
\DeclareMathOperator{\SU}{SU}
\DeclareMathOperator{\D}{D}
\DeclareMathOperator{\G}{G}

\DeclareMathOperator{\U}{U}

\DeclareMathOperator{\spc}{sp}

\DeclareMathOperator{\Arg}{Arg}

\DeclareMathOperator{\End}{End}

\DeclareMathOperator{\diag}{diag}

\DeclareMathOperator{\id}{id}

\DeclareMathOperator{\spn}{span}

\DeclareMathOperator{\Ad}{Ad}

\DeclareMathOperator{\im}{\mathrm{im}}

\newcommand{\ca}[1]{\ensuremath{\mathcal{#1}}}

\newcommand{\norm}[1]{\ensuremath{ {\left\| #1 \right\|} }}

\newtheorem{proposition}{Proposition}[section]
\newtheorem{lemma}[proposition]{Lemma}

\newtheorem{theorem}[proposition]{Theorem}
\newtheorem{theoremrec}{Theorem}

\newtheorem{corollary}[proposition]{Corollary}
\theoremstyle{definition}

\newtheorem{remarks}[proposition]{Remarks}

\newtheorem{definition}[proposition]{Definition}
\newtheorem{example}[proposition]{Example}

\newtheorem{remark}[proposition]{Remark}

\numberwithin{equation}{section}

\newlength{\leftstackrelawd}
\newlength{\leftstackrelbwd}
\def\leftstackrel#1#2{\settowidth{\leftstackrelawd}%
{${{}^{#1}}$}\settowidth{\leftstackrelbwd}{$#2$}%
\addtolength{\leftstackrelawd}{-\leftstackrelbwd}%
\leavevmode\ifthenelse{\lengthtest{\leftstackrelawd>0pt}}%
{\kern-.5\leftstackrelawd}{}\mathrel{\mathop{#2}\limits^{#1}}}

\begin{document}

\title[Continuous spectrum-shrinking maps and applications to preserver problems]{Continuous spectrum-shrinking maps and applications to preserver problems}

\author{Alexandru Chirvasitu, Ilja Gogi\'{c}, Mateo Toma\v{s}evi\'{c}}

\address{A.~Chirvasitu, Department of Mathematics, University at Buffalo, Buffalo, NY 14260-2900, USA}
\email{achirvas@buffalo.edu}

\address{I.~Gogi\'c, Department of Mathematics, Faculty of Science, University of Zagreb, Bijeni\v{c}ka 30, 10000 Zagreb, Croatia}
\email{ilja@math.hr}

\address{M.~Toma\v{s}evi\'c, Department of Mathematics, Faculty of Science, University of Zagreb, Bijeni\v{c}ka 30, 10000 Zagreb, Croatia}
\email{mateo.tomasevic@math.hr}




\date{\today}

\begin{abstract}
  For a positive integer $n$ let $\mathcal{X}_n$ be either the algebra $M_n$ of $n \times n$ complex matrices, the set $N_n$ of all $n \times n$ normal matrices, or any of the matrix Lie groups $\mathrm{GL}(n)$, $\mathrm{SL}(n)$ and $\mathrm{U}(n)$. We first give a short and elementary argument that for two positive integers $m$ and $n$ there exists a continuous spectrum-shrinking map $\phi : \mathcal{X}_n \to M_m$ (i.e.\ $\mathrm{sp}(\phi(X))\subseteq \mathrm{sp}(X)$ for all $X \in \mathcal{X}_n$) if and only if $n$ divides $m$. Moreover, in that case we have the equality of characteristic polynomials $k_{\phi(X)}(\cdot) = k_{X}(\cdot)^\frac{m}{n}$ for all $X \in \mathcal{X}_n$, which in particular shows that $\phi$ preserves spectra. Using this we show that whenever $n \geq 3$, any continuous commutativity preserving and spectrum-shrinking map $\phi : \mathcal{X}_n \to M_n$ is of the form $\phi(\cdot)=T(\cdot)T^{-1}$ or $\phi(\cdot)=T(\cdot)^tT^{-1}$, for some $T\in \mathrm{GL}(n)$. The analogous results fail for the special unitary group $\mathrm{SU}(n)$ but hold for the spaces of semisimple elements in either $\mathrm{GL}(n)$ or $\mathrm{SL}(n)$. As a consequence, we also recover (a strengthened version of) \v{S}emrl's influential characterization of Jordan automorphisms of $M_n$ via preserving properties.
\end{abstract}

\maketitle


\section*{Introduction}

The present paper fits into the body of literature revolving around classifying maps between various matrix spaces which preserve spectra or various other invariants/properties. A few recollections will help appreciate the general flavor of the topic. 

Denote by $M_n$, $\GL(n)$ and $\U(n)$ the algebra of $n \times n$ complex matrices and the subgroups of invertible and unitary matrices therein respectively. It is well-known that any nonzero \emph{Jordan endomorphism} $\phi$ of $M_n$ (i.e.\ \cite[\S I.1]{jac_jord} a linear map $\phi:M_n\to M_n$ with $\phi(XY+YX) = \phi(X)\phi(Y) + \phi(Y)\phi(X)$ for all $X,Y \in M_n$) is  either a conjugation by an invertible matrix, or a conjugation composed with transposition:
\begin{equation}\label{eq:inner}
  \phi(\cdot) = T(\cdot)T^{-1} \qquad \text{or}   \qquad \phi(\cdot) = T(\cdot)^{t}T^{-1},
\end{equation}
for some $T \in \GL(n)$ (see e.g.\ \cite{Herstein,Semrl2}). There are numerous results in the ever-growing literature characterizing Jordan morphisms between matrix and more general (operator) algebras via suitable (linear on non-linear) preserving properties. One of the most prominent nonlinear preserver characterizations of Jordan automorphisms of $M_n$ is \v{S}emrl's \cite[Theorem~1.1]{Semrl}:

\begin{theoremrec}[\v{S}emrl]\label{thm:Semrl}
  Let $\phi : M_n \to M_n, n \ge 3$, be a continuous commutativity and spectrum preserving map. Then there exists $T \in \GL(n)$ such that $\phi$ is of the form \eqref{eq:inner}.
\end{theoremrec}

Recall that a map $\phi: M_n \to M_n$ preserves
\begin{itemize}[wide]
\item commutativity if
  \begin{equation*}
    \forall X,Y\in M_n
    \quad:\quad
    [X,Y]:=XY-YX=0
    \xRightarrow{\quad}
    [\phi(X),\phi(Y)]=0;
  \end{equation*}

\item and spectra if $\spc(\phi(X))=\spc(X)$ for all $X \in M_n$. 
\end{itemize}
The first version of this result was in fact formulated by Petek and \v Semrl in \cite{PetekSemrl}, with an additional assumption that $\phi$ preserves commutativity in both directions or rank-one matrices. For other related results and further generalizations we refer to \cite{GogicPetekTomasevic,GOGIC2025129497,Petek-HM,Petek-TM} and the references within.


\smallskip

We now turn to the main objective of the paper. First recall that if $\mathcal{A}$ and $\mathcal{B}$ are unital (Banach) algebras, then a map $\phi: \mathcal{A} \to \mathcal{B}$ is said to be \emph{spectrum-shrinking} if $\spc(\phi(x))\subseteq \spc(x)$ for all $x \in \mathcal{A}$. Motivated by Theorem \ref{thm:Semrl}, the famous Kaplansky-Aupetit question (asking whether Jordan epimorphisms between unital semisimple Banach algebras can be characterized as linear spectrum-shrinking surjective maps, see e.g.\ \cite{Aupetit,BresarSemrl,Kaplansky}) and the results of \cite{LiTsaiWangWong}, the question arises of whether an arbitrary continuous spectrum-shrinking map $\phi : M_n \to M_m$ (if such exists) automatically preserves spectra.

The issue of spectrum shrinking versus spectrum preservation becomes even more salient given its connection to Kaplansky's problem: finding sufficient conditions that will ensure an invertibility-preserving map (between complex unital algebras, in our context) is a Jordan morphism. Linear unital maps $\phi$ between unital Banach algebras
\begin{itemize}[wide]
\item preserve \emph{and reflect} invertibility, i.e.
  \begin{equation*}
    \phi(x)\text{ invertible }
    \xLeftrightarrow{\quad}
    x\text{ invertible},
  \end{equation*}
  precisely when they preserve spectra;
  
\item and preserve invertibility exactly when they \emph{shrink} spectra instead.
\end{itemize}
Any characterization / classification result involving spectrum preservation will have a potential spectrum-shrinking counterpart, but the literature seems to suggest such improvements are not always routine: 

\begin{itemize}[wide]
\item \cite[Theorem 1.1]{MR1866032} and \cite[Theorem 2]{MR832991} characterize linear surjections between bounded-endomorphism algebras of Banach spaces which both preserve and reflect. The former, in particular, gives a very short proof of the main result;

\item By contrast, a partial analogue \cite[Theorem 1.1]{MR1311919} requiring only invertibility preservation is substantially more involved.

\item Similarly non-trivial are the techniques involved in characterizing invertibility-preserving linear maps between matrix algebras in \cite[Theorem 2]{zbMATH05719630} or \cite[Theorem 4.1]{MR2736150}, say. 
\end{itemize}

Returning to the question of the extent to which spectrum shrinking entails automatic spectrum preservation, it turns out this is indeed the case even upon restricting the domain of $\phi$ to $\GL(n)$, $\U(n)$, or any number of other topological matrix spaces of interest. Corollary \ref{cor:sp.shrk.conn.conf.sp} below provides a sample of possibilities for what those matrix spaces can be:

\begin{theorem}\label{thintro:many.spcs}
  Let $m,n\in \bZ_{\ge 1}$ and take for $\cX_n\le M_n$ any one of the following spaces:
  \begin{itemize}[wide]
  \item $\cX_n=M_n$ or the subspace of diagonalizable matrices;
  \item $\cX_n=\GL(n)$ or the subspace of diagonalizable invertible matrices;
  \item $\cX_n=\SL(n)$ or the subspace of diagonalizable determinant-1 matrices;
  \item $\cX_n=\U(n)$;
  \item $\cX_n=N_n$ -- the subspace of normal $n\times n$ matrices. 
  \end{itemize}
  Then there exists a continuous spectrum shrinker
  \begin{equation*}
    \phi : \cX_n \to M_m
  \end{equation*}
  if and only $n$ divides $m$, in which case we have the equality
  \begin{equation*}
    k_{\phi(X)}
    =
    (k_{X})^{\frac mn}
    ,\quad
    \forall  X\in \cX_n,
  \end{equation*}
  where $k_{\bullet}$ denotes the characteristic polynomial of $\bullet$. 
\end{theorem}
By Remarks \ref{rem:Theorem main} the analogous result fails for the subspace $H_n$ of $n\times n$ self-adjoint matrices and the special unitary group $\SU(n)$, $n \geq 2$, as both spaces admit a continuous eigenvalue selection.

\smallskip 

Section \ref{se:cls} focuses instead on results classifying continuous spectrum-shrinking maps defined on various matrix spaces. A paraphrased aggregate of Theorems \ref{th:main-result-gln} and \ref{th:ss} reads as follows.

\begin{theorem}\label{thintro:clsf}
  Let $n\in \bZ_{\ge 3}$ and $\phi : \cX_n\to M_n$ a continuous, commutativity-preserving and spectrum-shrinking map for a subset $\cX_n\subseteq M_n$.

  If $\ca{X}_n\in \{\GL(n), \SL(n), \U(n), N_n\}$ or consists of the diagonalizable matrices in either $\GL(n)$ or $\SL(n)$ then $\phi$ is either a conjugation or a conjugation composed with transposition.
\end{theorem}

Our proof of the semisimple branch of Theorem \ref{thintro:clsf} does not rely on \v{S}emrl's theorem \ref{thm:Semrl}. In fact, a simple continuity argument recovers Theorem \ref{thm:Semrl} via Theorems \ref{thintro:many.spcs} and \ref{thintro:clsf}.

The techniques employed in proving Theorems \ref{thintro:many.spcs} and \ref{thintro:clsf} involve very little direct computation, being in part algebraic, in part reliant on the topology of the various matrix spaces involved, and in part operator-theoretic: in analyzing a qualitatively new candidate
\begin{equation}\label{eqintro:sns}
  SNS^{-1}
  \xmapsto{\quad}
  S^{-1}NS
  ,\quad
  \forall \text{ normal }N\in N_n\text{ and positive }S\in \GL(n).
\end{equation}
(to show that it in fact meets all of the requirements save for continuity) we make crucial use of the celebrated \emph{Putnam-Fuglede theorem} \cite[Theorem]{zbMATH03133061}.


There is some hope of extending the main results, and several directions present themselves for doing so: other compact matrix Lie groups in place of $\U(n)$, various infinite-dimensional versions of the results, etc. In part, this forms the object of future work based on adjacent techniques.

\subsection*{Acknowledgments}

We are grateful for P. \v{S}emrl's many insightful and illuminating comments, tips and pointers to the relevant literature on Kaplansky-type problems.  

\section{Leveraging spectrum shrinking into spectrum preservation}\label{se:autopres}

To spell out a convenient common generalization, we need some notation.

First of all, by $\bC^{\times}$ we denote the subgroup of nonzero complex numbers. Now let $n \in \bZ_{\ge 1}$. As usual, we identify vectors in $\C^n$ with the corresponding column-matrices. Similarly, depending on context, we also identify vectors in $\C^n$ with corresponding diagonal matrices and vice versa. For a matrix $X \in M_n$ we denote by $k_X(x) := \det(x I_n-X)$ its characteristic polynomial ($I_n$ is the identity matrix in $M_n$). We naturally identify the symmetric group $S_n$ with the $n \times n$ permutation matrices, so that $S_n \le \GL(n)$. The convention pertaining to permutations is that $\sigma\in S_n\le \GL(n)$ operates by
\begin{equation*}
  e_j
  \xmapsto{\quad\sigma\quad}
  e_{\sigma(j)}
  ,\quad
  (e_1, \ldots, e_n)\text{ the standard basis of }\bC^n
\end{equation*}
or, equivalently, the left action on (column) size-$n$ vectors is 
\begin{equation}\label{eq:sn.act}
  (x_1, \ldots, x_n)
  \xmapsto{\quad\sigma\quad}
  (x_{\sigma^{-1}(1)}, \ldots , x_{\sigma^{-1}(n)}).
\end{equation}
The symbol $\Ad_{\G}$ denotes the conjugation (or \emph{adjoint}) action of a subgroup $\G\le \GL(n)$ on $M_n$, so that for $S \subseteq M_n$,
\begin{equation*}
  \Ad_{\G}S=\{gxg^{-1} \ : \ x \in S,\ g \in \G\}.
\end{equation*}

Given a subset $L\subseteq \C^n$, by $\Delta_{L}$ we denote the subset of $L$ that consists of elements with at least two equal coordinates:
\begin{equation*}
  \Delta_{L}
  :=
  \left\{(x_1, \ldots, x_n)\in L \ : \ x_j=x_k\text{ for some }j\ne k\right\}.
\end{equation*}
If $V$ is a subspace of the algebra $T^{+}(n)$ of $n \times n$ strictly upper-triangular matrices, denote by $T_{L,V} = T_{L,V}(n)$ the space of upper-triangular matrices with diagonal in $L$ and strictly upper-triangular component in $V$:

\begin{equation*}
  T_{L,V}:=\{\diag(\lambda_1, \ldots, \lambda_n) + v \ : \ (\lambda_1, \ldots, \lambda_n) \in L, \ v \in T^{+}(n)\}.
\end{equation*}

In the statement of Theorem \ref{th:sp.shrk.conn.conf.sp} below we refer to the conjugation action of the group $S_n\le \GL(n)$ of permutation matrices on the set $L$ of tuples, identified with that of diagonal matrices. $S_n$ will then also act on $\Delta_L$, on the complement $L\setminus \Delta_L$ and  hence, finally, also on the space of connected components of $L\setminus \Delta_L$. 

\begin{theorem}\label{th:sp.shrk.conn.conf.sp}
  Let $n\in \bZ_{\ge 1}$, $\G\le \GL(n)$ a closed connected subgroup, $V\le T^{+}(n)$ a linear subspace and $L\subseteq \bC^n$ a subset such that
  \begin{itemize}[wide]
  \item $L\setminus \Delta_L$ is dense in $L$;
  \item and the isotropy groups of the connected components of $L\setminus \Delta_L$ in $\G\cap S_n$ are transitive on $\{1,2,\ldots,n\}$.
  \end{itemize}
  Then for an arbitrary $m\in \bZ_{\ge 1}$ there exists a continuous spectrum shrinker
    \begin{equation*}
  \phi :   \Ad_{\G}T_{L,V} \to M_m
  \end{equation*}
  if and only $n$ divides $m$ and in that case we have the equality of characteristic polynomials
  \begin{equation}\label{eq:k.is.pow}
    k_{\phi(X)}
    =
    (k_{X})^{\frac mn}
    ,\quad
   \forall  X\in \Ad_{\G}T_{L,V}.
  \end{equation}
\end{theorem}

The proof of Theorem \ref{th:sp.shrk.conn.conf.sp} is (perhaps surprisingly) both elementary and short; we provide it below, after first recording a number of specializations to matrix spaces that are naturally of interest. For a subspace $M\le M_n$ we write $M_{ss}\subseteq M$ for the subspace of \emph{semisimple} (i.e.\ diagonalizable) matrices therein.

\begin{corollary}\label{cor:sp.shrk.conn.conf.sp}
  The conclusion of Theorem \ref{th:sp.shrk.conn.conf.sp} holds for continuous spectrum-shrinking maps $\phi : \cX_n\to M_m$ in any of the following cases.
  \begin{enumerate}[(a)]
  \item\label{item:cor:sp.shrk.conn.conf.sp:mn} $\cX_n=M_n$ or $(M_{n})_{ss}$; 
  \item\label{item:cor:sp.shrk.conn.conf.sp:gl} $\cX_n=\GL(n)$ or $\GL(n)_{ss}$;
  \item\label{item:cor:sp.shrk.conn.conf.sp:sl} $\cX_n=\SL(n)$ or $\SL(n)_{ss}$;
  \item\label{item:cor:sp.shrk.conn.conf.sp:u} $\cX_n=\U(n)$;
  \item\label{item:cor:sp.shrk.conn.conf.sp:norm} $\cX_n=N_n$. 
  \end{enumerate}
\end{corollary}
\begin{proof}
  This is a simple matter of specifying how to fit the various cases into the framework of Theorem \ref{th:sp.shrk.conn.conf.sp} by setting the parameters $\G$, $L$ and $V$.
  
  \begin{enumerate}[(I),wide]
  \item \textbf{: The group $\G$.} For \ref{item:cor:sp.shrk.conn.conf.sp:mn}, \ref{item:cor:sp.shrk.conn.conf.sp:gl} and \ref{item:cor:sp.shrk.conn.conf.sp:sl} it is $\G:=\GL(n)$, while in the other two cases it is $\G:=\U(n)$.

  \item \textbf{: The subspace $V\le T^{+}(n)$.} All of $T^{+}(n)$ in \ref{item:cor:sp.shrk.conn.conf.sp:mn}, \ref{item:cor:sp.shrk.conn.conf.sp:gl} and \ref{item:cor:sp.shrk.conn.conf.sp:sl} for arbitrary (as opposed to semisimple) matrices, and $V:=\{0\}$ in all other cases

  \item \textbf{: The subset $L\subseteq \bC^n$.} It is $L:=\bC^n$ in \ref{item:cor:sp.shrk.conn.conf.sp:mn}, $(\bC^{\times})^n$ in \ref{item:cor:sp.shrk.conn.conf.sp:gl},
    \begin{equation*}
      \left\{(\lambda_1, \ldots, \lambda_n)\in \bC^n\ :\ \prod_{j=1}^n \lambda_j=1\right\}
      \text{ in \ref{item:cor:sp.shrk.conn.conf.sp:sl}},
    \end{equation*}
    $(\bS^1)^n$ in \ref{item:cor:sp.shrk.conn.conf.sp:u} and again $\bC^n$ in \ref{item:cor:sp.shrk.conn.conf.sp:norm}. 
  \end{enumerate}
  
  In cases \ref{item:cor:sp.shrk.conn.conf.sp:mn}, \ref{item:cor:sp.shrk.conn.conf.sp:gl}, \ref{item:cor:sp.shrk.conn.conf.sp:sl} and \ref{item:cor:sp.shrk.conn.conf.sp:norm}, where $L$ is a connected complex algebraic variety and $\Delta_L$ a closed algebraic subset thereof, the complement $L\setminus \Delta_{L}$ is connected \cite[Proposition 8.3]{zbMATH06468787} so there is nothing further to check with regard to the transitivity hypothesis.
  
  In case \ref{item:cor:sp.shrk.conn.conf.sp:u} $L\setminus \Delta_{L}$ is precisely the $n^{th}$ \emph{configuration space} \cite[Definition 1.1]{zbMATH05785888} $\cC^n(\bS^1)$ consisting of $n$-tuples of distinct modulus-$1$ complex numbers. As we will see in Lemma \ref{le:un.cycl.conn}, $\cC^n(\bS^1)$ is disconnected as soon as $n\ge 3$. The same result, however, shows that the isotropy groups of the connected components are the (transitive) conjugates of the subgroup generated by the cycle $(1\ 2\ \cdots\ n)$.
\end{proof}  

In the following statement we denote by $\braket{\eta}$ the cyclic subgroup generated by an element $\eta\in S_n$ and by $\pi_0(X)$ the space of components of $X$; the space $X=\cC^n(\bS^1)$ we are concerned with being a manifold, components and path components coincide \cite[Propositions 4.23 and 4.26]{lee_top-mfld_2e_2011}. 

\begin{lemma}\label{le:un.cycl.conn}
  Let $n\in \bZ_{\ge 1}$. The symmetric group $S_n$ acts transitively on the space $\pi_0(\cC^n(\bS^1))$ of components of the circle's $n^{th}$ configuration space, and the isotropy groups are the conjugates
  \begin{equation*}
    g\braket{\eta}g^{-1}\le S_n
    ,\quad
    g\in S_n
    ,\quad
    \eta:=(1\ 2\ \cdots\ n).
  \end{equation*}
\end{lemma}
\begin{proof}
  Consider $n$ distinct points $z_1, \ldots,z_n\in \bS^1$. Traversing the circle counterclockwise will enumerate the points as
  \begin{equation*}
    z_{\tau(1)},\ z_{\tau(2)},\ldots,  z_{\tau(n)} 
  \end{equation*}
  for some permutation $\tau$, uniquely determined up to right multiplication by powers of $\eta$. Associating the class of $\tau$ in $S_n/\braket{\eta}$  to the component in question provides a map 
  \begin{equation}\label{eq:assoc.coset.class}
    \pi_0(\cC^n(\bS^1))
    \xrightarrow[\quad\text{equivariant for the $S_n$-actions on the (co)domain}\quad]{\quad}
    S_n/\braket{\eta}.
  \end{equation}
  That map is surjective because the $S_n$-action on its codomain $S_n/\braket{\eta}$ is transitive, and the conclusion will follow once we observe that \eqref{eq:assoc.coset.class} is also injective.

  It will suffice to prove that the preimage of $\braket{\eta}\in S_n/\braket{\eta}$ is a singleton (the argument is not materially different for other cosets). In other words, the goal is to show that
  \begin{equation}\label{eq:cc.zis}
    \left\{
      (z_1, \ldots, z_n)\in \cC^n(\bS^1)\ :\ \text{$z_j$ are ordered counterclockwise increasingly with $j$}
    \right\}
  \end{equation}
  is (path-)connected. To that end, note that a counterclockwise rotation allows us to assume $z_1=1\in \bS^1$, and the resulting subset of \eqref{eq:cc.zis} is
  \begin{equation*}
    \left\{
      \left(\exp(2\pi i x_1),\ \cdots,\ \exp(2\pi i x_n)\right)
      \ :\
      0=x_1<x_2<\cdots<x_n<2\pi
    \right\}.
  \end{equation*}
  This is the image through the continuous map $\bR^n\xrightarrow{\exp(2\pi i \bullet)^n}(\bS^1)^n$ of the (path-connected) interior of an $(n-1)$-simplex.
\end{proof}

\begin{remarks}\label{rem:Theorem main}
  \begin{enumerate}[(1),wide]
  \item If $m=rn$ for some $r \in \Z_{\ge 1}$, then an apparent class of continuous maps $\phi : \mathcal{X}_n \to M_m=M_r(M_n)$ which satisfy \eqref{eq:k.is.pow} is given by
    \begin{equation*}
      \phi(X)=S(X) \begin{pmatrix} X \otimes I_p & 0 \\ 0 & X^t  \otimes I_q\end{pmatrix} S(X)^{-1},
    \end{equation*} 
    where $S:\ca{X}_n \to \GL(m)$ is a continuous function and $p,q \in \Z_{\ge 0}$ are such that $p+q=r$. 
  \item The ``obvious analogue'' of Corollary \ref{cor:sp.shrk.conn.conf.sp} does not hold for the corresponding maps $\phi : H_n \to M_m$ defined on the real subspace $H_n$ of $n \times n$ self-adjoint matrices, as for any $m \in \bZ_{\ge 1}$ we can define $\phi(X):=\lambda_{\max}(X) I_m$, where $\lambda_{\max}(X)$ denotes the largest eigenvalue of $X\in H_n$ (which continuously depends on $X$).

  \item An analogous remark applies to $\cX_n:=\SU(n)$, the special unitary group: Corollary \ref{cor:sp.shrk.conn.conf.sp} does not hold for it, because for any $n \in \bZ_{\ge 1}$ there exists a continuous eigenvalue selection $\varphi : \SU(n) \to \bS^1$. As before, for any $m \in \bZ_{\ge 1}$ the assignment $X \mapsto \varphi(X)I_m$ defines a continuous, commutativity-preserving and spectrum-shrinking map $\SU(n) \to M_m$.
    
    To see that there is a continuous eigenvalue selection $\SU(n)\to \bS^1$, first identify (per \cite[Proposition IV.2.6]{btd_lie_1985}) the space $\SU(n)/\Ad_{\SU(n)}$ of conjugacy classes with the quotient $\bT/S_n$ of the maximal torus
    \begin{equation*}
      \bT:=\left\{(\lambda_1, \ldots, \lambda_n)\in \left(\bS^1\right)^n\ :\ \prod_{j=1}^n \lambda_j=1\right\}
    \end{equation*}
    by the \emph{Weyl group} $S_n$ of $\SU(n)$ \cite[Theorem IV.3.3]{btd_lie_1985}. Further, \cite[Lemmas 1 and 2]{MR210096} show that
    \begin{equation*}
      (x_1,\ldots, x_n)
      \xmapsto{\quad}
      \left(\exp(2\pi i x_1),\ldots, \exp(2\pi i x_n)\right)
    \end{equation*}
    implements a homeomorphism between
    \begin{equation*}
      F:=\left\{(x_1, \ldots, x_n)\in \bR^n\ :\ \sum_{j=1}^n x_j=0,\ x_1\le x_2\le \cdots\le x_n\le x_1+1\right\}
    \end{equation*}
    and $\bT/S_n$. We can now simply take our continuous eigenvalue selection to be
    \begin{equation*}
      \SU(n)
      \xrightarrowdbl{\quad}
      \SU(n)/\Ad_{\SU(n)}
      \cong
      \bT/S_n
      \cong
      F
      \ni
      (x_1, \ldots, x_n)
      \xmapsto{\quad}
      \exp(2\pi i x_1)
      \in \bS^1.
    \end{equation*}
  \end{enumerate}
\end{remarks}




One consequence of Corollary \ref{cor:sp.shrk.conn.conf.sp} is a generalization of the well-known fact that for $n \geq 2$ there is no (global) continuous selection of eigenvalues of matrices in $M_n$ (i.e.\ there does not exist a continuous map $\phi : M_n\to \C$ such that $\phi(X) \in \spc(X)$ for all $X \in M_n$, see e.g.\ \cite[Example~2]{LiZhang} and Proposition \ref{pr:nbhd.eig.sel.m} below for a more elaborate statement):

\begin{corollary}
  Let $m,n \in \bZ_{\ge 1}$ and let $\mathcal{X}_n$ be any one of the spaces of Corollary \ref{cor:sp.shrk.conn.conf.sp}.
  \begin{itemize}[wide]
  \item[(a)] If $n$ does not divide $m$, there does not exist a continuous spectrum-shrinking map $\phi: \mathcal{X}_n\to M_m$.
  \item[(b)] If $n$ divides $m$, then any continuous spectrum-shrinking map $\phi: \mathcal{X}_n\to M_m$ is in fact a spectrum preserver.
  \end{itemize}
\end{corollary}


The second direct consequence is the following strengthening of Theorem \ref{thm:Semrl}:
\begin{corollary}\label{cor:spectrum-shrinking}
  Let $\phi : M_n \to M_n, n \ge 3$, be a continuous  commutativity-preserving and spectrum-shrinking map. Then there exists $T \in\mathrm{GL}(n)$ such that $\phi$ is of the form \eqref{eq:inner}.
\end{corollary}

In the proofs of Theorems \ref{th:sp.shrk.conn.conf.sp} and \ref{th:main-result-gln} we shall use the following simple lemma:
\begin{lemma}\label{le:topological-lemma}
  Let $X$ and  $Y$ be topological spaces such that $X$ is connected and $Y$ is Hausdorff. Let $n\in \bZ_{\ge 1}$ and suppose that $g,f_1,\ldots,f_n : X \to Y$ are continuous functions such that for each $x \in X$ we have
  \begin{itemize}[wide]
  \item $g(x) \in \{f_1(x),\ldots,f_n(x)\}$,
  \item $f_j(x) \ne f_k(x)$ for all $1 \le j \ne k \le n$ and all $x\in X$.
  \end{itemize}
  Then there exists $1 \leq j \leq n$ such that $g=f_j$.
\end{lemma}
\begin{proof}
  The sets $\{x \in X : g(x) = f_1(x)\},\ldots, \{x \in X : g(x) = f_n(x)\}$ are closed and form a partition of $X$. By the connectedness of $X$, exactly one of them is nonempty and hence equal to all of $X$.
\end{proof}


\begin{proof}[Proof of Theorem \ref{th:sp.shrk.conn.conf.sp}]
  The ``only if'' part of the theorem is covered by Remark \ref{rem:Theorem main}(1).

\smallskip

Denote by $\bC_{\le m}[x]$ the space of complex polynomials of degree $\leq m$, equipped with the unique Hausdorff linear topology, compatible with that of $\bC$. The function
  \begin{equation*}
    \Ad_{\G}T_{L,V}
    \xrightarrow{\quad k_{\phi(\cdot)}\quad}
    \bC_{\le m}[x]
  \end{equation*}
  is continuous, and on an element with spectrum $(\lambda_1, \ldots, \lambda_n)$ takes one of the finitely many values $\prod_{j=1}^n (x-\lambda_j)^{k_j}$ for non-negative integers $k_1, \ldots, k_n$ with sum $m$. It is enough to prove the claimed equality \eqref{eq:k.is.pow} for matrices of the form
  \begin{equation}\label{eq:t.simple.spec}
    T=g \left(\diag(\lambda_1, \ldots, \lambda_n)+v\right)g^{-1}
    ,\quad
    g\in \G,\ (\lambda_1, \ldots, \lambda_n) \in L':=L\setminus \Delta_{L},\ v\in V,
  \end{equation}
  as such matrices (whose eigenvalues are all distinct) form a dense subspace of $\Ad_{\G}T_{L,V}$. By Lemma \ref{le:topological-lemma} for $T$ as in \eqref{eq:t.simple.spec} the characteristic polynomial $k_{\phi(T)}$ depends only on $(\lambda_1, \ldots, \lambda_n)\in L'$: $g$ and $v$ range over the (path-)connected spaces $\G$ and $V$ respectively. Furthermore, by the selfsame Lemma \ref{le:topological-lemma}, for every connected component $L'_0$ of $L'$ we have  
  \begin{equation}\label{eq:kil0}
    k_{\phi(\diag(\lambda_1, \ldots, \lambda_n))}
    =
    \prod_{j=1}^n (x-\lambda_j)^{k_{j,L'_0}},
    \quad
   \forall (\lambda_1, \ldots, \lambda_n)\in L'_0.
  \end{equation}
  Now, because we are furthermore assuming the isotropy group of $L'_0$ in $S_n\cap \G$ is transitive and $k_{\phi(\cdot)}$ is invariant under $\G$-conjugation, the exponents of \eqref{eq:kil0} must in fact all be equal. Indeed, consider $1\le  \ell<\ell'\le n$ and $\sigma \in S_n\cap \G$ fixing $L'_0$ with $\sigma(\ell)=\ell'$. We then have
  \begin{equation*}
    \begin{aligned}
      (x-\lambda_{\ell'})^{k_{\ell',L'_0}}
      \cdot
      \prod_{j\ne \ell'}^n (x-\lambda_j)^{k_{j,L'_0}}
      &=
        \prod_{j=1}^n (x-\lambda_j)^{k_{j,L'_0}}\\
      &=
        k_{\phi(\diag(\lambda_1, \ldots, \lambda_n))}
        \quad\text{by \eqref{eq:kil0}}\\
      &=
        k_{\phi(\sigma^{-1} \diag(\lambda_1, \ldots, \lambda_n)\sigma)}
        \quad\text{by the $\G$-invariance of $k_{\phi(\cdot)}$}\\
      &=
        k_{\phi(\diag(\lambda_{\sigma(1)}, \ldots, \lambda_{\sigma(n)}))}
      \quad\text{by \eqref{eq:sn.act}}\\
      &=
        \prod_{j=1}^n \left(x-\lambda_{\sigma(j)}\right)^{k_{j,L'_0}}
        \quad\text{by \eqref{eq:kil0} again, because $\sigma$ fixes $L'_0$}\\
      &=
        (x-\lambda_{\ell'})^{k_{\ell,L'_0}}
        \cdot
        \prod_{j\ne \ell}^n (x-\lambda_{\sigma(j)})^{k_{j,L'_0}}
    \end{aligned}    
  \end{equation*}
  and hence $k_{\ell',L'_0} = k_{\ell,L'_0}$.
\end{proof}

\begin{remarks}\label{res:loc.spc.shrk}
  An essential feature of $\ca{X}_n$ in Corollary \ref{cor:sp.shrk.conn.conf.sp} is the nonexistence of a continuous eigenvalue selection $\ca{X}_n \to \C$. 
  \begin{enumerate}[(1),wide]


  \item\label{item:res:loc.spc.shrk:unl} Even though Corollary \ref{cor:sp.shrk.conn.conf.sp} holds for $\ca{X}_n = \U(n)$, it does not hold for any of its open dense subsets
    \begin{equation*}
      \U(n)_{\lambda}:=\{X \in \U(n) \ : \ \lambda \notin \spc(X)\},
    \end{equation*}
    where $\lambda \in \bS^1$. Indeed, for each $X \in \U(n)_{\lambda}$ choose $\varphi(X) \in \spc(X)$ with the largest value of the corresponding (continuous) branch of the argument, defined on $\C\setminus \{t\lambda : t \geq 0\}$. Then as before, if $n \geq 2$, the assignment $\U(n)_{\lambda} \to M_m$, $X \mapsto \varphi(X)I_m$, defines a continuous proper spectrum shrinker. In particular, as $\{\U(n)_\lambda \ : \ \lambda \in \bS^1\}$ forms on open cover of $\U(n)$, any unitary matrix admits a local continuous eigenvalue selection.
    
  \item\label{item:res:loc.spc.shrk:mnloc} Fix any norm $\norm{\cdot}$ on $M_n$. If $\ca{X}_n\in \{M_n, \GL(n)\}$, we can slightly strengthen the statement of Corollary \ref{cor:sp.shrk.conn.conf.sp} for maps defined on $B_{\norm{\cdot}}(0_n,r)\cap \ca{X}_n$, for any $r>0$ (instead of the whole $\ca{X}_n$). Indeed, if $\psi :  B_{\norm{\cdot}}(0_n,r)\cap \ca{X}_n \to M_m$ is a continuous spectrum-shrinking map, then this follows by applying Corollary \ref{cor:sp.shrk.conn.conf.sp} to the map
    \begin{equation*}\phi : \ca{X}_n \to M_m, \qquad \phi(X) := \frac{1+\norm{X}}{r}\psi\left(\frac{r X}{1+\norm{X}}\right),\end{equation*}
    which satisfies the same properties. 

  \end{enumerate}
\end{remarks}

Remark \ref{res:loc.spc.shrk}\ref{item:res:loc.spc.shrk:mnloc} implies that no neighborhood of $0_n\in M_n$ admits a continuous eigenvalue selector (by contrast to the parallel situation in $\U(n)$, per Remark \ref{res:loc.spc.shrk}\ref{item:res:loc.spc.shrk:unl}). The following result answers the natural question of which matrices share this property.

\begin{proposition}\label{pr:nbhd.eig.sel.m}
  Let $n\in \bZ_{\ge 1}$. A matrix in either $M_n$ or $\GL(n)$ has a neighborhood admitting a continuous eigenvalue selection if and only if it has at least one simple eigenvalue. 
\end{proposition}
\begin{proof}
  The statement for $M_n$ will of course recover that for $\GL(n)$, the latter being an open subspace of the former. We thus focus on $M_n$.
  
  \begin{enumerate}[(I),wide]
  \item \textbf{: $(\Leftarrow)$} Suppose $\lambda$ is a simple spectrum point for $X\in M_n$. By the continuity of both characteristic polynomials and spectra \cite[Theorem 3]{Newburgh}, for small $\varepsilon>0$ and a sufficiently small neighborhood $\cU\ni X$ in $M_n$ every $Y\in \cU$ has a unique eigenvalue
    \begin{equation*}
      \lambda_Y
      \in
      \left\{z\in \bC\ :\ |z-\lambda|<\varepsilon\right\}.
    \end{equation*}
    This provides us with desired (plainly continuous) selector: $Y\mapsto \lambda_Y$ for $Y\in \cU$. 
    
  \item \textbf{: $(\Rightarrow)$} Assume now that all spectrum points of $X\in M_n$ have algebraic multiplicity $\ge 2$, and that we have a continuous eigenvalue selector
    \begin{equation*}
      \cU\ni Y
      \xmapsto{\quad}
      \lambda_Y
      \in \bC
      ,\quad
      \cU\ni X\text{ a neighborhood}. 
    \end{equation*}
    Restricting attention to matrices leaving invariant the (at least $2$-dimensional, by assumption) \emph{generalized $\lambda$-eigenspace} \cite[\S 9.4]{bronson_matr}
    \begin{equation*}
      \bigcup_{m\in \bZ_{\ge 0}}\ker\left(\lambda I_n-X\right)^m
    \end{equation*}
    of $X$, we can assume the spectrum $\spc(X)$ is a singleton. Translation by $\lambda I_n\in M_n$ further reduces the problem to nilpotent operators $X$. Because moreover every neighborhood of a nilpotent operator contains another such with a single \emph{Jordan block} \cite[Definition 3.1.1]{hj_mtrx}, conjugation turns the problem into its instance for
    \begin{equation*}
      X
      :=
      J_n(0)
      :=
      \begin{pmatrix}
        0&1&\phantom{-}&\phantom{-}&\phantom{-}\\
        \phantom{-}&0&1&\phantom{-}&\phantom{-}\\
        \phantom{-}&\phantom{-}&\ddots&\ddots&\phantom{-}\\
        \phantom{-}&\phantom{-}&\phantom{-}&0&1\\
        \phantom{-}&\phantom{-}&\phantom{-}&\phantom{-}&0
      \end{pmatrix}.
    \end{equation*}
    This case, in turn, is easily handled as in \cite[Example~2]{LiZhang} (the current problem's version for $n=2$): for any neighborhood $\cU\ni 0$ in $\bC$ the $z$-dependent matrix
    \begin{equation*}
      X_z
      :=
      \begin{pmatrix}
        0&1&\phantom{-}&\phantom{-}&\phantom{-}\\
        \phantom{-}&0&1&\phantom{-}&\phantom{-}\\
        \phantom{-}&\phantom{-}&\ddots&\ddots&\phantom{-}\\
        \phantom{-}&\phantom{-}&\phantom{-}&0&1\\
        z&\phantom{-}&\phantom{-}&\phantom{-}&0
      \end{pmatrix}
    \end{equation*}
    with characteristic polynomial $x^n\pm z$ does not admit a continuous eigenvalue selector.  \qedhere
  \end{enumerate}
\end{proof}


\section{Classification results}\label{se:cls}

One of the present section's main results is the following preserver result: 


\begin{theorem}\label{th:main-result-gln}
  Let $\ca{X}_n\in \{\GL(n), \SL(n), \U(n), N_n\}$, $n \geq 3$ and $\phi : \ca{X}_n \to M_n$ a continuous commutativity-preserving and spectrum-shrinking map.

  There exists $T \in \GL(n)$ such that $\phi$ is of the form \eqref{eq:inner}.
\end{theorem}



We gather some auxiliary remarks and results building up to a proof of Theorem \ref{th:main-result-gln}. Proposition \ref{pr:alg.cpct.tor} is essentially a variant of \cite[Lemma 2.1]{Semrl} applicable to either algebraic or unitary \emph{tori}, i.e.\ $(\bC^{\times})^n$ or $(\bS^1)^n$ respectively. The proof of the the aforementioned result adapts easily to arbitrary (as opposed to unitary) invertible matrices, but we provide a short alternative in that case as well for completeness.


\begin{proposition}\label{pr:alg.cpct.tor}
  For $n\in \bZ_{\ge 1}$ denote by $\G\leq \GL(n)$ the group of arbitrary, determinant-1, or unitary diagonal matrices. 
  
  A continuous commutativity- and spectrum-preserving map $\phi : \G \to \GL(n)$ is conjugation by some $T\in \GL(n)$. 
\end{proposition}
\begin{proof}
  Consider a fixed diagonal matrix $X=\diag(x_1, \ldots, x_n)\in \G$ with distinct eigenvalues $x_j$.  Its image $\phi(X)$ is a conjugate thereof (by spectrum preservation), so we may as well assume $\phi(X)=X$ for simplicity (and hence $\phi(\G)$ consists of diagonal matrices by commutativity preservation). It follows that $\phi(Y)$ simply permutes the entries of $Y$ for any $Y\in \G$ with distinct eigenvalues:
  \begin{equation*}
    \left(\forall Y=\diag(y_1, \ldots, y_n)\in \G\right)
    \left(\exists \tau\in S_n\right)
    \quad:\quad
    \phi(Y) = \diag(y_{\tau(1)},\ldots, y_{\tau(n)}).
  \end{equation*}
  As before, we identify the groups of invertible diagonal and unitary diagonal matrices in $\GL(n)$ with $(\bC^{\times})^n$ and $(\bS^1)^n$, respectively. 
  \begin{enumerate}[(I),wide]
  \item \textbf{: $\G=(\bC^{\times})^n$ or the determinant-1 tuples therein.}  That $\tau$ is the identity for all $Y\in \G$ with distinct eigenvalues follows from: the continuity of $\phi$, the fact that the space of simple-spectrum elements of $\G$ is connected (per \cite[Proposition 8.3]{zbMATH06468787}, as remarked in the proof of Corollary \ref{cor:sp.shrk.conn.conf.sp}), Lemma \ref{le:topological-lemma}, and the density of such matrices $Y$ in $\G$.
    
  \item \textbf{: $\G=(\bS^1)^n$.} By Lemma \ref{le:un.cycl.conn}, $\cC^n(\bS^1)$ is connected only for $n=1,2$, when the preceding argument applies; we henceforth assume $n \geq 3$. What Lemma \ref{le:topological-lemma} still delivers is a map sending a component $C\in \pi_0(\cC^n(\bS^1))$ to the permutation $\tau=\tau_C\in S_n$ satisfying
    \begin{equation*}
      \phi(Y) = \diag(y_{\tau(1)},\ldots, y_{\tau(n)})
      ,\quad
      \forall Y\in C,
    \end{equation*}
    and hence a function
    \begin{equation*}
      \begin{tikzpicture}[>=stealth,auto,baseline=(current  bounding  box.center)]
        \path[anchor=base] 
        (0,0) node (l) {$S_n$}
        +(4,.5) node (ul) {$S_n/\braket{(1\ 2\ \cdots\ n)}$}
        +(8,.8) node (ur) {$\pi_0(\cC^n(\bS^1))$}
        +(12,0) node (r) {$S_n$}
        ;
        \draw[->] (l) to[bend left=6] node[pos=.5,auto] {$\scriptstyle $} (ul);
        \draw[->] (ul) to[bend left=6] node[pos=.5,auto] {$\scriptstyle \text{\eqref{eq:assoc.coset.class}}$} node[pos=.5,auto,swap] {$\scriptstyle \cong$} (ur);
        \draw[->] (ur) to[bend left=6] node[pos=.5,auto,swap] {$\scriptstyle $} (r);
        \draw[->] (l) to[bend right=6] node[pos=.5,auto,swap] {$\scriptstyle \sigma\xmapsto{\quad}\tau_{\sigma}$} (r);
      \end{tikzpicture}
    \end{equation*}
    Fix a transposition $\sigma=(j\ j')$ and consider convergent tuple sequences
    \begin{equation*}
      C
      \supset
      \textbf{y}_m=(y_{m,i})_{i=1}^n
      \xrightarrow[\quad m\quad]{}
      \textbf{y}=(y_i)_i
      \xleftarrow[\quad m\quad]{}
      \sigma\textbf{y}_m=(y_{m,\sigma^{-1}(i)})_{i=1}^n
      \subset
      C'
      :=
      \sigma C
    \end{equation*}
    for a connected component $C$ of $\cC^n(\bS^1)$, with the common limit
    \begin{equation*}
      \textbf{y}\in \overline{C}\cap \overline{C'}\subseteq (\bS^1)^n      
    \end{equation*}
    having a single equality between two components (necessarily the $j^{th}$ and $(j')^{th}$). Writing $\tau:=\tau_C$ and $\tau':=\tau_{C'}$, we have
    \begin{equation*}
      (y_{\tau(i)})_i = (y_{\sigma\tau'(i)})_i
      \ \xRightarrow{\quad}\ 
      y_{\ell} = y_{\sigma\tau'\tau^{-1}(\ell)}
      ,\quad \forall 1\le \ell\le n
      \ \xRightarrow{\quad}\ 
      \tau'\in \left\{\tau,\ \sigma\tau\right\}.
    \end{equation*}
    In conclusion:
    \begin{equation}\label{eq:trnsp.lipsch}
      \left(\forall \theta\in S_n\right)
      \left(\forall\text{ transposition }\sigma\in S_n\right)      
      \quad:\quad
      \tau_{\sigma\theta}\in \left\{\tau_{\theta},\ \sigma\tau_{\theta}\right\}.
    \end{equation}
    The map $\tau_{\bullet}$, by its very definition, is constant on left cosets of $\braket{(1\ 2\ \cdots\ n)}$:
    \begin{equation}\label{eq:const.on.cosets}
      \left(\forall \theta\in S_n\right)
      \left(\forall s\in \bZ\right)
      \quad:\quad
      \tau_{\theta \eta^s} = \tau_{\theta\eta^s\theta^{-1}\cdot \theta} = \tau_{\theta}
      \quad\text{for}\quad
      \eta=(1\ 2\ \cdots\ n).
    \end{equation}

    
    Observe, next, that for every $\theta\in S_n$ and every transposition $\sigma$ there is some $s\in \bZ/n$ and decomposition
    \begin{equation}\label{eq:cdots.fix}
      \theta \eta^s\theta^{-1} = \sigma' \sigma
      ,\quad
      \text{$\sigma'\in S_n$ fixes one of the symbols involved in $\sigma$}. 
    \end{equation}
    It is enough to check this for $\eta$ itself, for conjugation by $\theta\in S_n$ will not affect the claim. Considering the $n$ symbols modulo $n$ (so that $1=n+1$, $0=n$, etc.), an arbitrary transposition $\sigma$ can be written as $(j\ j+s)$. The permutation $\sigma':=\eta^s(j\ j+s)$ fixes $j+s$; this means precisely that $\eta^s = \sigma'(j\ j+s)$ with $\sigma'$ fixing $j+s$, i.e.\ \eqref{eq:cdots.fix} holds.

    Now fix $\theta\in S_n$, a transposition $\sigma$, and a decomposition \eqref{eq:cdots.fix}. The latter reads
    \begin{equation*}
      \theta \eta^s\theta^{-1} = \sigma_1\cdots \sigma_{\ell} \sigma
    \end{equation*}
    for a decomposition $\sigma'=\prod_{i=1}^{\ell}\sigma_i$ as a product of transpositions. Because $\sigma'$ fixes one of the symbols moved by $\sigma$, no word of the form
    \begin{equation*}
      \sigma_{i_1} \cdots \sigma_{i_k}\sigma
      ,\quad
      1\le i_1<i_2<\cdots<i_k\le \ell
    \end{equation*}
    can be trivial, and hence \eqref{eq:trnsp.lipsch} and \eqref{eq:const.on.cosets} imply that $\tau_{\theta}=\tau_{\sigma\theta}$. This holds for \emph{arbitrary} $\theta\in S_n$ and transpositions $\sigma$, so we are done: $\tau_{\bullet}$ is constant.  \qedhere
  \end{enumerate}
\end{proof}

We also recast \cite[Lemma 2.4]{Semrl} in a form suitable for work with normal or unitary operators. We denote by $\bG=\bG V$ the \emph{Grassmannian} of the ambient space $V:=\bC^n$  of column vectors on which the algebra $M_n$ acts: the collection of all of its subspaces, topologized as
\begin{equation*}
  \bG V=\coprod_{0\le d\le n}\bG(d,V)
  ,\quad
  \bG(d,V):=\left\{\text{$d$-dimensional subspaces of $V$}\right\}
  \quad\text{\cite[Example 1.36]{lee_smth-mfld_2e_2013}}.
\end{equation*}
Furthermore, we equip the space $V=\C^n$ with its standard inner product. For a subspace $W \in \bG$ by $P_W$ we denote the orthogonal projection onto $W$. 
\begin{lemma}\label{le:def.latt.mor}
  Let $n\in \Z_{\ge 1}$ and $\phi: \U(n) \to \GL(n)$ a continuous commutativity- and spectrum-preserving map. The assignment $\Psi : \bG \to \bG$ given by 
  \begin{equation}\label{eq:w2tw}
   \Psi(W) := \ker(\lambda
    I_n-\phi(U)), \text{ for any } \lambda \in \bS^1 \text{ and }  U\in \U(n)\text{ with }\ker(\lambda I_n-U)=W,
  \end{equation}
  is a well-defined continuous map which preserves both dimensions and inclusion relations.
  
\end{lemma}
\begin{proof}  
  Fix a subspace $W \le \C^n$ and consider the unitary operator
  \begin{equation*}
    U_W :=2P_W-I_n= \id_W \oplus \,(-\id_{W^\perp}) \in \U(n).
  \end{equation*}  
  Suppose that $U \in \U(n)$ and $\lambda\in \bS^1$ satisfy
  \begin{equation*}
    \ker(\lambda I_n - U) = W.
  \end{equation*}
  To show that $\Psi$ is well-defined, it suffices to show
  \begin{equation}\label{eq:well-defined}
      \ker(\lambda I_n - \phi(U)) = \ker(I_n - \phi(U_W)).
  \end{equation}
  Indeed, let $S \in \U(n)$ be a unitary matrix whose columns consist of the orthonormal eigenbasis for $U$. It easily follows that
  \begin{equation*}
    U, U_W \in S(\bS^1)^nS^*.
  \end{equation*}
  Proposition \ref{pr:alg.cpct.tor} applied to the maximal torus $\G:=S(\bS^1)^nS^*$ implies that there exists $T_{\G} \in \GL(n)$ such that
  \begin{equation*}
    \phi(X) = T_{\G} X T_{\G}^{-1}
    ,\quad
    \forall
    X\in \G.
  \end{equation*}
  In particular we have
  \begin{align*}
    \phi(U) = T_{\G} UT_{\G}^{-1} \implies \ker(\lambda I_n-\phi(U)) & = \ker (T_{\G}(\lambda I_n - U)T_{\G}^{-1}) = T_{\G}(\ker(\lambda I_n - U)) \\
    &= T_{\G} W.
  \end{align*}
  Exactly the same calculation gives $\ker(I_n-\phi(U_W)) = T_{\G} W$ which yields  \eqref{eq:well-defined}, so that
  \begin{equation}\label{eq:uulw}
    \Psi(W)=
    \ker(I_n-\phi(U_{W}))=T_{\G} W.
  \end{equation}

That $\Psi$ preserves dimensions follows from the map's expression \eqref{eq:uulw}, as does inclusion preservation: whenever $W\le W'$ there is a maximal torus $\G$ in $\U(n)$ containing both $U_W$ and $U_{W'}$, so that $\Psi(W)=T_{\G} W\le T_{\G} W'=\Psi(W')$.

  It remains to verify the continuity of $\Psi$. For a fixed $0 \leq d \leq n$ let $\cU_{d}\subset M_n$ be the collection of diagonalizable matrices $X \in M_n$ such that $\dim\ker(I_n - X) = d$. It is straightforward to check that the map
  \begin{equation*}
    \cU_{d} \ni X
    \xmapsto{\quad\Lambda\quad}
    \ker( I_n - X)
    \in \bG(d,V)
  \end{equation*}
  is continuous (e.g.\ \cite[Proposition 13.4]{salt_divalg}). As $U_{W} \in \cU_{d}$ varies continuously with $W \in \bG(d,V)$, it follows that $\Psi$ is continuous as a composition
  \begin{equation*}
    \bG(d,V) \ni W
    \xmapsto{\quad}
    \underbracket{U_{W}}_{\in \cU_{d}}
    \xmapsto{\quad}
    \underbracket{\phi(U_{W})}_{\in \cU_{d}}
    \xmapsto{\quad\Lambda\quad}
    \ker(I_n - \phi(U_{W}))=\Psi(W)\in  \bG(d,V)
  \end{equation*}
  of continuous maps.
\end{proof}

The map of Lemma \ref{le:def.latt.mor} is somewhat more than dimension-preserving: if $V$ and $W$ are (complex) vector spaces, recall \cite[Definition 2.1]{zbMATH01747827} that a map $\Psi:\bP V \to \bP W$ is a \emph{morphism between projective spaces} if $\Psi$ respects sums of lines in the following sense:
\begin{equation}\label{eq:is.p.mor}
  \ell \le \ell'\bigvee \ell''
  \xRightarrow{\quad}
  \Psi\ell\le \Psi\ell'\bigvee\Psi\ell''
  ,\quad
  \forall \text{ lines }\ell,\ \ell',\ \ell''\in \bP V,
\end{equation}
where `$\bigvee$' denotes the supremum operation (i.e.\ sum) in the \emph{subspace lattice} \cite[\S 1.2]{pank_wign} of $V$. To prove that $\Psi$ is a morphism in this sense (or rather restricts to one on $\bP V\subset \bG$) we need a version of the \emph{Fundamental Theorem of Projective Geometry} (appearing in various guises as \cite[Theorem 3.1]{zbMATH01747827}, \cite[Theorem 2.3]{pank_wign}, etc.). 

For subspaces $W,W'\le V\cong \bC^n$ write 
$W\bigobot W'$ to denote the fact that the orthogonal projections $P_W$ and $P_{W'}$ commute; this is equivalent to the orthogonal complements of $W\cap W'$ in $W$ and $W'$ being orthogonal, or to the fact that there is a maximal torus $\G\le \U(n)$ leaving both $W$ and $W'$ invariant. Note first:

\begin{lemma}\label{le:pres.ops.orth}
  The map $\Psi$ of Lemma \ref{le:def.latt.mor} preserves lattice operations for spaces in relation $\bigobot$.
\end{lemma}
\begin{proof}
  This is precisely for the reason just pointed out in defining $\bigobot$, i.e.\ the existence of $\G$: for $T_G$ attached to $\G$ as in \eqref{eq:uulw} we have
  \begin{equation*}    
    \Psi(W+W')=T(W+W')=TW+TW'=\Psi(W)+\Psi(W'). 
  \end{equation*}
\end{proof}

The following result will thus apply to the map $\Psi$ constructed in Lemma \ref{le:def.latt.mor}. 

\begin{proposition}\label{pr:pres.sup}
  Let $V$ be an $n$-dimensional complex Hilbert space for $n \in \bZ_{\ge 3}$, $\bG:=\bG(V)$ its Grassmannian. Any continuous self-map  $\Psi: \bG \to \bG$  that preserves dimensions and lattice operations for spaces in relation $\bigobot$ is induced by an invertible linear or conjugate-linear operator $T$ on $V$, in the sense that
  \begin{equation*}
    \forall W\in \bG
    \quad:\quad
    \Psi(W)=TW. 
  \end{equation*}
\end{proposition}
\begin{proof}
  There are several stages to the proof. First of all, as $\Psi$ preserves dimensions, it restricts to a self-map $\Psi|_{\bP V}:\bP V \to \bP V$. 

  \begin{enumerate}[(I),wide]
  \item \textbf{Conclusion, assuming $\Psi|_{\bP V}$ is a projective space endomorphism.} Suppose, in other words, we have proven \eqref{eq:is.p.mor}. The image $\Psi(\bP V)$ of the projective space $\bP V\subset \bG$ cannot be contained in a line of $\bP V$, for by assumption the restriction $\Psi|_{\bP V}$ sends lines through the elements of an orthogonal basis to distinct $n=\dim V\ge 3$ (even linearly independent) lines. As we are assuming (in the present portion of the proof) that $\Psi|_{\bP V} : \bP V \to \bP V$ is an endomorphism of $\bP V$, it follows from \cite[Theorem 3.1]{zbMATH01747827} that $\Psi|_{\bP V}$ is implemented by a \emph{semilinear} endomorphism $T: V \to V$:
    \begin{equation*}
      \Psi(\ell) =\spn_{\bC}T\ell, \quad \forall \ell \in  \bP V,
    \end{equation*}  
    where $T$ is an additive map for which there exists a ring endomorphism $\alpha\in \End(\bC)$ such that
    \begin{equation*}
      T(\mu v)=\alpha(\mu)Tv,\quad \forall \mu \in \bC,\ v\in V.
    \end{equation*}
    An arbitrary subspace $W \in \bG(d,V)\subset \bG$ ($1 \leq d \leq n$) can be decomposed as $W=\ell_1 \oplus \cdots \oplus \ell_d$ for some mutually orthogonal lines $\ell_j \in \bP V$, so the assumed compatibility with lattice operations under $\bigobot$ yields 
    \begin{equation}\label{eq:Psi(W)}
      \Psi(W) = \Psi(\ell_1) + \cdots +\Psi(\ell_d)= \spn_{\bC} TW.
    \end{equation}    
    Obviously $T$ is injective, as $\Psi(\bP V)\subseteq \bP V$. As (by assumption) $\Psi$ is continuous, so is   $\alpha$. Indeed (as argued in \cite[\S 3, p.133]{Semrl}), consider orthogonal non-zero vectors $v,w\in V$. Lattice compatibility under $\bigobot$ then implies that the lines
    \begin{equation*}
      \bC Tv = \Psi(\bC v)
      \quad\text{and}\quad
      \bC Tw = \Psi(\bC w)
    \end{equation*}
    are distinct. In particular, $Tv$ and $Tw$ are linearly independent. The map
    \begin{equation*}
      \bC\ni \mu
      \xmapsto{\quad}
      \Psi(\bC(v+\mu w))
      =
      \bC\left(Tv+\alpha(\mu)Tw\right)
      \in \bP V
    \end{equation*}
    is continuous, and with it $\alpha$ by the noted linear independence of $Tv$ and $Tw$. That continuity forces $\alpha\in \left\{\id,\ \overline{\bullet}\right\}$, so that by \eqref{eq:Psi(W)} we have 
    \begin{equation*}
     \forall W\in \bG
        \quad:\quad \Psi(W) = TW,
    \end{equation*}
    for $T$ either in $\GL(V)$ or invertible \emph{conjugate}-linear.
    
  \item \textbf{$\Psi|_{\bP V}$ is a projective space endomorphism.} The dimension preservation of $\Psi$ already constrains the nature of possible counterexamples: the assumed compatibility with $\bigobot$ implies that $\Psi$ preserves inclusion, so for lines $\ell$, $\ell'$ and $\ell''$ with $\ell\le \ell'\bigvee\ell''$ we have
    \begin{equation*}
      \text{lines }\Psi(\ell),\ \Psi(\ell'),\ \Psi(\ell'')\le \text{2-plane }\Psi\left(\ell\bigvee\ell'\right).
    \end{equation*}
    All is thus well provided $\Psi$ is injective on $\bP V$ (for in that case, if $\ell'\ne \ell''$, the distinct lines $\Psi(\ell')$ and $\Psi(\ell'')$ would have to span the right-hand 2-plane in the above equation). We can always reduce the problem to the case $n=3$, as we will: if $\Psi(\ell)=\Psi(\ell')$ for distinct lines then consider a 3-dimensional subspace $V'\le V$ containing $\ell'\bigvee\ell''$, and substitute
    \begin{itemize}[wide]
    \item $V'$ for $V$;
    \item and $R\circ\Psi$ for $\Psi$, where $R\in \GL(V)$ sends $\Psi(V')$ back to $V'$ and we identify $R$ with the operator $W\mapsto RW$ it induces on $\bG$.
    \end{itemize}
    
    Suppose, then, that $\Psi(\ell)=\Psi(\ell')$ for $\ell\ne \ell'\in \bP V$ and $\dim V=3$. We claim that in this case $\Psi$ is constant on all of $\bP \pi$, where $\pi:=\ell\bigvee \ell'$, contradicting the fact that by assumption $\Psi$ sends orthogonal lines to linearly independent lines. The projective line $\bP\pi$ being connected, it will be enough to argue that
    \begin{equation}\label{eq:psipsil}
      \Psi^{-1}\left(\Psi(\ell)\right)\cap \bP\pi
      =
      \left\{\ell''\le \pi\ :\ \Psi(\ell'')=\Psi(\ell)=\Psi(\ell')\right\}
      \subseteq
      \bP\pi
    \end{equation}
    is open (for it is already closed by the continuity of $\Psi$, and non-empty because it contains $\ell$ and $\ell'$ themselves). Without loss of generality, this amounts to showing that along with $\ell'$ the set in question contains an entire neighborhood thereof in $\bP\pi$. Write
    \begin{equation*}
      \pi':=\spn\left\{\ell',\ \pi^{\perp}\right\},
    \end{equation*}
    so that $\pi\bigobot \pi'$. We proceed in stages. 
    
    \begin{enumerate}[(a),wide]
    \item \textbf{: The line $\ell'$ belongs to the interior of $\Psi^{-1}(\Psi(\ell))\cap \bP \pi'$ regarded as a subspace of $\bP \pi'$.} Suppose not. There is then a sequence
      \begin{equation*}
        \bP \pi'\setminus \Psi^{-1}(\Psi(\ell))
        \ni
        \ell_k
        \xrightarrow[\quad k\quad]{\quad\text{convergence in $\bG$}\quad}
        \ell'.
      \end{equation*}
      because by assumption $\Psi(\ell_k)\ne \Psi(\ell')=\Psi(\ell)$ for all $k$, we have (as observed earlier in the proof)
      \begin{equation*}
        \forall k\quad:\quad
        \Psi\left(\ell_k\bigvee \ell'\right)
        =
        \Psi(\ell_k)\bigvee \Psi(\ell')
        \xlongequal[]{\quad\Psi(\ell')=\Psi(\ell)\quad}
        \Psi(\ell_k)\bigvee \Psi(\ell)
        =
        \Psi\left(\ell_k\bigvee \ell\right).
      \end{equation*}
      Now, $\ell_k\bigvee\ell'=\pi'$ for all $k$ on the one hand and $\ell_k\bigvee\ell\xrightarrow[k]{}\pi$ on the other, so  $\Psi(\pi)=\Psi(\pi')$. On the other hand
      \begin{equation*}
        \pi\bigobot \pi'
        \xRightarrow{\quad}
        \Psi\left(\pi\bigvee\pi'\right)
        =
        \Psi(\pi) \bigvee \Psi(\pi'),
      \end{equation*}
      and we have a contradiction. 
      
    \item \textbf{: The line $\ell'$ belongs to the interior of \eqref{eq:psipsil}.} Simply apply the preceding step with altered parameters: $\ell'$ is as before, but substitute any
      \begin{equation*}
        \ell'' \in \Psi^{-1}(\Psi(\ell))\cap \bP \pi'\setminus \{\ell'\}
      \end{equation*}
      for $\ell$ and interchange the roles of $\pi$ and $\pi'$.  \qedhere
    \end{enumerate}
    
  \end{enumerate}
\end{proof}

\setcounter{case}{0}

\begin{proof}[Proof of Theorem \ref{th:main-result-gln}]
 First of all, by Corollary \ref{cor:sp.shrk.conn.conf.sp}, $\phi$ preserves spectra.

  \begin{enumerate}[(I),wide]
  \item\label{item:th:main-result-gln:gl}  \textbf{: $\ca{X}_n = \mathrm{GL}(n)$.}
    By the continuity of the spectral radius $\rho(\cdot)$ on $M_n$ (see e.g.\ \cite[Theorem 3]{Newburgh}), the map
    \begin{equation*}
      F : M_n \to \mathrm{GL}(n), \qquad F(X):= X + (1+\rho(X))I_n
    \end{equation*}
    is well-defined and continuous. Consider the map \begin{equation*}\psi : M_n \to M_n, \qquad \psi(X) := \phi(F(X)) - (1 + \rho(X))I_n.\end{equation*} Clearly, $\psi$ is well-defined, continuous, commutativity and spectrum preserver. Therefore, Theorem \ref{thm:Semrl} implies that there exists $T \in \mathrm{GL}(n)$ such that $\psi$ is of the form \eqref{eq:inner}. By passing to the map $T^{-1}\phi(\cdot)T$ or $(T^{-1}\phi(\cdot)T)^t$, without loss of generality, we can assume that $\psi$ is the identity map on $M_n$. In particular, we have
    \begin{equation}\label{eq:F(X)}
      \phi(F(X))=F(X), \quad \forall X \in M_n.
    \end{equation}

    By Proposition \ref{pr:alg.cpct.tor}, $\phi$ operates as conjugation by an invertible matrix on every maximal algebraic torus of $\GL(n)$ (i.e.\ conjugate of the subgroup $\D(n)\le \GL(n)$ of diagonal matrices):
    \begin{equation*}
      \left(\forall S\in \GL(n)\right)
      \left(\exists R\in \GL(n)\right)
      \left(\forall D\in \D(n)\right)
      \quad:\quad
      \phi(SDS^{-1})
      =
      RSDS^{-1}R^{-1}.
    \end{equation*}
    The consequence \cite[Corollary 2.2]{Semrl} follows:
    \begin{equation}\label{eq:intertw.poly}
      \left(\forall X\in \GL(n)\right)
      \left(\forall\text{ polynomials }p\ :\ p(X)\in \GL(n)\right)
      \quad:\quad
      \phi(p(X))=p(\phi(X)).
    \end{equation}
    Now, for a fixed $X\in \GL(n)$ we have $F(X)=p_X(X)$ for the polynomial
    \begin{equation*}
      p_X(\lambda):=\lambda+(1+\rho(X)),
    \end{equation*}
    and hence (as $\rho(\phi(X))=\rho(X)$)
    \begin{align*}
      \phi(X)+(1+\rho(X))I_n
      &=
        p_X(\phi(X))
        \overset{\text{\eqref{eq:intertw.poly}}}{=}
        \phi(p_X(X))
        =
        \phi(F(X))
        \overset{\text{\eqref{eq:F(X)}}}{=}
        F(X)\\
      &=
        X+(1+\rho(X))I_n.
    \end{align*}
    We are now done: $\phi(X)=X$, as desired.

  \item\label{item:th:main-result-gln:sl} \textbf{: $\ca{X}_n = \mathrm{SL}(n)$.} Let 
    \begin{equation*}
      \sqrt[n]{\cdot}: \C\setminus \{x \in \R \ : \ x \le 0\} \to \C\setminus \{x \in \R \ : \ x \le 0\}, \qquad \sqrt[n]{z}=\sqrt[n]{|z|}e^{\frac{i \Arg z}{n}}
    \end{equation*}
    be the principal branch of the complex $n$-th root (which is a continuous map). Denote
    \begin{equation*}
      \mathrm{GL}(n)_*:= \{X \in \mathrm{GL}(n) \ : \ \det X \ne -1\}.
    \end{equation*}
    The map
    \begin{equation}\label{eq:psi.conj.root.det}
      \zeta : \mathrm{GL}(n)_* \to \mathrm{GL}(n)_*, \qquad \zeta(X) = \sqrt[n]{\det X}\cdot \phi\left(\frac1{\sqrt[n]{\det X}} X\right)
    \end{equation}
    is a continuous, spectrum- and commutativity-preserving extension of $\phi$ to $\GL(n)_* \supseteq \SL(n)$. Step \ref{item:th:main-result-gln:gl} goes through essentially without change for $\mathrm{GL}(n)_*$: Corollary \ref{cor:sp.shrk.conn.conf.sp} applies in the variant of its case \ref{item:cor:sp.shrk.conn.conf.sp:gl} with
    \begin{equation*}
      L:=\left\{(\lambda_1, \ldots, \lambda_n)\in (\bC^{\times})^n\ :\ \prod_{j=1}^n \lambda_j\ne -1\right\};
    \end{equation*}
    the complement $L\setminus \Delta_L$ is connected, as in the earlier proof. We conclude that there exists $T \in \mathrm{GL}(n)$ and $\circ \in \{\id, (\cdot)^t\}$ such that
    \begin{equation*}
      \zeta(X) = TX^{\circ}T^{-1}, \quad \forall X \in \mathrm{GL}(n)_*
    \end{equation*}
    and consequently 
    \begin{equation*}
      \phi(X) = TX^\circ T^{-1}, \quad \forall X \in \mathrm{SL}(n).
    \end{equation*}
    
  \item\label{item:th:main-result-gln:u} \textbf{: $\ca{X}_n = \mathrm{U}(n)$.} Denote by $V\cong \bC^n$ the ambient space on which all operators act. The map $\Psi$ constructed in Lemma \ref{le:def.latt.mor} meets the requirements of Proposition \ref{pr:pres.sup} by Lemma \ref{le:pres.ops.orth}, so that there exists an invertible $T: V \to V$, either linear or conjugate-linear, such that
    \begin{align}\label{eq:psiw}
      \forall W\in\bG
      \quad:\quad  
      & \Psi(W) =\ker(\lambda I_n -\phi(U)) =  T(\ker(\lambda I_n -U)) \\ & \text{ for any } \lambda \in \bS^1 \text{ and } U \in \U(n) \text{ with }  W=\ker(\lambda I_n-U) \nonumber
    \end{align}
  This means the following.

    \begin{itemize}[wide]
    \item When $\alpha=\id$, we have the first option in \eqref{eq:inner}:
      \begin{equation*}
        \phi(U)=TUT^{-1}
        ,\quad
        \forall U\in \U(n).
      \end{equation*}
      Indeed,  for an arbitrary $U \in \U(n)$,  Proposition \ref{pr:alg.cpct.tor} and \eqref{eq:psiw}  say that $\phi(U)$ is the diagonalizable operator with the same spectrum as $U$ and whose $\lambda$-eigenspace (for any $\lambda\in \bS^1$) is the image through $T$ of the $\lambda$-eigenspace of $U$. In other words, $\phi(U)$ is precisely $TUT^{-1}$. 

    \item When $\alpha$ is conjugation, the analogous argument proves that 
      \begin{equation*}
        \phi(U)=TU^{-1}T^{-1}
        ,\quad
        \forall U\in \U(n)
      \end{equation*}
      instead: once more $\phi(U)$ is the diagonalizable operator whose $\lambda$-eigenspaces are the respective images through $T$ of those of $U$, while conjugation by $T$ turns $\lambda$-eigenspaces into $\overline{\lambda}$-eigenspaces respectively (for $\lambda\in \bS^1$ so that $\overline{\lambda}=\lambda^{-1}$) by anti-linearity. 
      
      This second variant corresponds to the right-hand branch of \eqref{eq:inner}: the conjugate-linear involution $e_i\xmapsto{J}e_i$ attached to the orthonormal basis we have already fixed implicitly in treating operators as matrices has the property that 
      \begin{equation*}
        \forall X\in M_n
        \quad:\quad
        JXJ^{-1}=\overline{X}
      \end{equation*}
      (where for $X=(x_{ij})$, $\overline{X}=(\overline{x_{ij}})$) and hence
      \begin{equation*}
        \phi(U) = TU^{-1}T^{-1} = (TJ)U^t(TJ)^{-1}
      \end{equation*}
      because $U$ is unitary (so that $JU^tJ^{-1}=\overline{U^t}=U^{-1}$). 

    \end{itemize}

  \item\label{item:th:main-result-gln:n} \textbf{: $\ca{X}_n = N_n$.} The conclusion certainly holds for the restriction $\phi|_{\U(n)}$ by the preceding step of the proof, so we can assume $\phi|_{\U(n)}=\id$. Proposition \ref{pr:alg.cpct.tor} shows that $\phi$ acts as a conjugation on every maximal abelian subalgebra of $N_n$, i.e.\ conjugate of the space of diagonal matrices. Commuting normal matrices are simultaneously unitarily diagonalizable \cite[Theorem 2.5.5]{hj_mtrx}, so $\phi$ restricted to any commuting family is additive. An arbitrary normal matrix is a linear combination of commuting unitary matrices (obvious for diagonal matrices, which suffices), hence the conclusion:
  \begin{equation*}
    \phi|_{\U(n)}=\id|_{\U(n)}
    \xRightarrow{\quad}
    \phi=\id_{N_n}. 
  \end{equation*}

  \end{enumerate}
  This concludes the proof of the theorem. 
\end{proof}

\begin{remark}\label{re:diff.fund.thm}
  The proof of \cite[Theorem 1.1]{Semrl} in \cite[\S 3]{Semrl} also uses the \cite[Theorem 3.1]{zbMATH01747827}, as step \ref{item:th:main-result-gln:u} of the above proof does. The upshot, however, is qualitatively different:

  \begin{itemize}[wide]
  \item In \ref{item:th:main-result-gln:u} above, the two possibilities for the endomorphism $\alpha$ (identity and conjugation) precisely correspond to the two types of map listed in \eqref{eq:inner}. 

  \item The proof of \cite[\S 3]{Semrl}, on the other hand, branches into the two options of \eqref{eq:inner} earlier; the branch $\phi=\Ad_T$ then rules out $\alpha=$ conjugation after applying \cite[Theorem 3.1]{zbMATH01747827}.
  \end{itemize}
\end{remark}

There is also a variant of Theorem \ref{th:main-result-gln} applicable to semisimple matrices, relying on the unitary/normal branch of the earlier result. Recall that an `ss' subscript denotes the collection of semisimple matrices in whatever space the subscript adorns. Additionally, a `$\ge 0$' subscript denotes the subspace of \emph{positive} \cite[Definition I.2.6.7]{blk} matrices (i.e.\ what \cite[Definition 4.1.9]{hj_mtrx} calls \emph{positive semi-definite}) of an ambient matrix set. 

\begin{theorem}\label{th:ss}
  Let $\ca{X}_n\in \{\GL(n)_{ss}, \SL(n)_{ss}\}$, $n \geq 3$. The continuous commutativity-preserving and spectrum-shrinking maps $\phi : \ca{X}_n \to M_n$ are precisely those of the form \eqref{eq:inner}.
\end{theorem}

Taking this for granted for the moment, we return to the announced alternative approach it affords to the $\GL$ and $\SL$ cases of Theorem \ref{th:main-result-gln} (not employing \cite[Theorem~1.1]{Semrl}, and hence recovering the latter as a byproduct).

\begin{proof}[Alternative proof of Theorem \ref{th:main-result-gln}: $\GL(n)$ and $\SL(n)$]

  Simply observe that the $\phi$ in question, defined on $\cX_n$, restrict to such on $\cX_{n,ss}$. The conclusion follows from the density of $\cX_{n,ss}\subseteq \cX_{n}$ for $\cX_n\in \left\{\GL(n),\ \SL(n)\right\}$. 
\end{proof}

We now settle into proving Theorem~\ref{th:ss}. A preliminary result will reduce the pool of maps to consider.

\begin{proposition}\label{pr:still.theta}
  Let $\ca{X}_n\in \{\GL(n)_{ss}, \SL(n)_{ss}\}$, $n \geq 3$. The continuous commutativity-preserving and spectrum-shrinking maps $\phi : \ca{X}_n \to M_n$ are expressible as compositions of the form
  \begin{equation}\label{eq:comp3maps}
    (\cdot)^t\circ \Ad_T\circ \Theta
    ,\quad T\in \GL(n),
  \end{equation}
  where 
  \begin{itemize}[wide]
  \item $\Theta$ is the involutory self-map of $\ca{X}_n$ defined by 
    \begin{equation}\label{eq:sns}
      SNS^{-1}
      \xmapsto{\quad\Theta\quad}
      S^{-1}NS
      ,\quad
      \forall
      N\in N_n\cap \ca{X}_n
      \text{ and }
      S\in \GL(n)_{\ge 0};
    \end{equation}
  \item and all of the components of \eqref{eq:comp3maps} are optional (i.e.\ any might be absent).
  \end{itemize}
\end{proposition}
As we will soon demonstrate, the ``exotic'' assignment $\Theta$ of \eqref{eq:sns} is indeed well-defined on $\ca{X}_n$ and preserves both spectra and commutativity. Because nevertheless a non-trivial argument establishes that it is not continuous on $\ca{X}_n$, it will be excluded from further consideration (see Proposition~\ref{pr:no.theta}). The proof of Proposition \ref{pr:still.theta} requires some preparation and auxiliary observations.

\begin{lemma}\label{le:all.no.transp}
  Under the hypotheses of Theorem \ref{th:ss} with $\cX_n=\GL(n)_{ss}$, there is $\circ\in \{\text{blank},\ t\}$ such that
  \begin{equation*}
    \left(\forall S\in \GL(n)\right)    
    \left(\exists T\in \GL(n)\right)
    \left(\forall U\in \U(n)\right)
    \quad:\quad
    \phi(SUS^{-1})=T(SUS^{-1})^{\circ}T.
  \end{equation*}
\end{lemma}
\begin{proof}
  The issue is quantification ordering: the unitary case of Theorem \ref{th:main-result-gln} shows that for each $S$ there are such $T=T_S$ and $\circ=\circ_S$, whereas we claim here that the \emph{same} $\circ$ will do for \emph{all} $S$. To see this, note that the closed subsets
  \begin{equation}\label{eq:s.circ}
    \left\{S\in \GL(n)\ :\ \circ_S=\circ\right\}\subseteq \GL(n)
    ,\quad
    \circ\in \{\text{blank},\ t\}
  \end{equation}
  partition the connected topological space $\GL(n)$. Indeed, the closure and disjointness claims both follow from the characterization of the sets \eqref{eq:s.circ} as
  \begin{equation*}
    \left\{
      S\in \GL(n)\ :\ \phi|_{\Ad_S \U(n)}
      \left[
      \begin{array}{l}        
        \text{preserves multiplication (if $\circ = \text{blank}$)}\\
        \text{reverses multiplication (if $\circ = t$)}        
      \end{array}
      \right.
    \right\}.
  \end{equation*}
  Multiplication preservation and reversal are both closed conditions by the continuity of $\phi$, and the two sets cannot overlap because multiplication and its opposite are distinct operations on the non-abelian conjugates of $\U(n)$.
\end{proof}

\begin{proposition}\label{pr:is.or.s2}
  Under the hypotheses of Theorem \ref{th:ss} with $\cX_n=\GL(n)_{ss}$, suppose $\phi|_{\U(n)}=\id_{\U(n)}$. We then have
  \begin{equation}\label{eq:le:is.or.s2:ids2}
    \left(\exists p\in \{0,\ -2\}\right)
    \left(\forall 0\le S\in \GL(n)\right)
    \quad:\quad
    \phi|_{\Ad_S \U(n)}
    =
    \Ad_{S^p}.
  \end{equation}
\end{proposition}
\begin{proof}
  The claim is that \eqref{eq:le:is.or.s2:ids2} holds with the same $p$ for all $S$. We will first prove a weaker claim, allowing possibly variable $p$ for different $S$. 
  
  \begin{enumerate}[(I),wide]
  \item \textbf{: For each $0\le S\in \GL(n)$ \eqref{eq:le:is.or.s2:ids2} holds for some $p=p_S\in \{0,-2\}$.} We take \eqref{eq:stu} for granted, along with its ancillary notation. First of all, for an arbitrary unitary $U \in \U(n)$ which commutes with $S$, \eqref{eq:stu} for $T=T_S$ implies
  \begin{equation*}
  U=\phi(U)=\phi(SUS^{-1})=TSUS^{-1}T^{-1}=TUT^{-1}.
  \end{equation*}
  Therefore, $T$ commutes with all unitary matrices commuting with $S$, or equivalently, $T$ belongs to the \emph{bicommutant} \cite[\S I.2.5.3]{blk} of (the positive) $S$. Hence, $T$ must be a polynomial in $S$ by the (finite-dimensional version of the) celebrated \emph{von Neumann bicommutant theorem} \cite[Theorem I.9.1.1]{blk}:
    \begin{equation}\label{eq:exists.ps}
      \exists p_S\in \bC[x]
      \quad:\quad
      T=p_S(S). 
    \end{equation}
    We may as well assume $S$ diagonal: $S=(\lambda_1, \ldots, \lambda_n)$, with the standard basis vectors $(e_1, \ldots, e_n)$ as eigenvectors (as usual, we identify diagonal matrices with the corresponding tuples). It will often also be convenient to assume the tuples $(\lambda_j)=(\lambda_1, \ldots, \lambda_n)$ we work with range over a conveniently chosen dense space (e.g.\ they are all distinct): by the continuity of the map
    \begin{equation*}
      \GL(n)_{\ge 0}
      \ni
      S
      \xmapsto{\quad}
      \Ad_S \U(n)\in \left\{\text{closed bounded subsets of $\GL(n)$}\right\}
    \end{equation*}
    for the topology on the codomain induced by the \emph{Hausdorff distance} \cite[Definition 7.3.1]{bbi} (attached to any metric topologizing $\GL(n)$), the sets
    \begin{equation}\label{eq:twops}
      \left\{0\le S\in \GL(n)\ :\ \phi|_{\Ad_S \U(n)}=\Ad_{S^p}\right\}\subseteq \GL(n)_{\ge 0}
      ,\quad
      p\in \{0,\ -2\}
    \end{equation}
    are closed. 
    
    Next, commutativity preservation of $\phi$ entails
    \begin{equation}\label{eq:uv.comm}
      \left(\forall U,V\in \U(n)\right)
      \ :\ 
      [U,\ SVS^{-1}]=0
      \xRightarrow{\quad}
      [\phi(U),\ \phi(SVS^{-1})]
      =
      [U,\ TSVS^{-1}T^{-1}]=0.
    \end{equation}
    
    We examine \eqref{eq:uv.comm} as applied to the following setup:
    \begin{itemize}[wide]
    \item $V$ is a unitary (hence also self-adjoint) involution with a one-dimensional $(-1)$-eigenspace spanned by $v=\sum_{j=1}^nc_j e_j$.
      
    \item $U$ is also a unitary involution, whose $(-1)$-eigenspace is spanned by $Sv=\sum_{j=1}^n c_j \lambda_j e_j$ and its orthogonal projection on $S(v^{\perp})$ (where $v^\perp$ is the orthogonal complement of $\{v\}$). Said $(-1)$-eigenspace is thus at most 2-dimensional. Generically, the dimension will be precisely 2: it is enough to assume the $S$ non-scalar (as discussed, there is no harm in doing so) and the $c_i$ all non-zero.
    \end{itemize}
    By fiat, $U$ leaves invariant the $\pm 1$-eigenspaces $\bC Sv$ and $S(v^{\perp})$ of $SVS^{-1}$ so it commutes with that operator. It then follows from \eqref{eq:uv.comm} that $U$ commutes with $\Ad_{TS}V$ as well.  
    
    In particular, the $(-1)$-eigenvector
      \begin{equation}\label{eq:definition of TSv}
        TSv
        =
        \sum_{j=1}^n c_j \lambda'_j e_j
        ,\quad
        \lambda'_j:=\lambda_j p(\lambda_j)
        ,\quad
        p=p_S\text{ as in \eqref{eq:exists.ps}}
      \end{equation}
      of $\Ad_{TS}V$ (unique up to scaling) is also an eigenvector of $U$, with eigenvalue $\pm 1$. We claim that in fact
      \begin{equation*}
        TSv\in \ker(I_n+U),
      \end{equation*}
      i.e.\ $TSv$ is a $(-1)$-eigenvector of $U$.  This is certainly true for $S=(\lambda_j)$ sufficiently close to the identity: $\Ad_{TS}V$ will then be close to $\Ad_SV$ for the latter is in turn close to a unitary and we are assuming $\phi=\id$ on unitaries. The $(-1)$-eigenspaces of $\Ad_SV$ and $\Ad_{TS}V$ will thus also be close in the projective space $\bP \bC^n$, so in particular the $(-1)$-eigenspace of $\Ad_{TS}V$ cannot be contained in the $1$-eigenspace of $U$. To extend this to all $S=(\lambda_j)$ under consideration (i.e.\ non-scalar diagonal positive matrices), consider the parameter space 
      \begin{equation}\label{eq:params}
        \cP
        :=
        \left\{\left((c_j),\ (\lambda_j)\right)\ :\ c_j\in \bC^{\times},\ \lambda_j>0\text{ not all equal}\right\}
      \end{equation}
      for our choices. $U$, $V$ and $v$ depend continuously on the $(c_j)$ and $(\lambda_j)=S$. We have maps
      \begin{equation*}
        \begin{tikzpicture}[>=stealth,auto,baseline=(current  bounding  box.center)]
          \path[anchor=base] 
          (0,0) node (l) {$\cP\ni ((c_j),\ (\lambda_j))$}
          +(5,1) node (u) {$\bC TSv\in \bP\bC^n=\bG(1,\bC^n)$}
          +(7,0) node (r) {$\ker(I_n+U)\in \bG(2,\bC^n)$}
          +(5,-1) node (d) {$\ker(I_n-U)\in \bG(n-2,\bC^n)$}
          ;
          \draw[|->] (l.north east) to[bend left=6] node[pos=.5,auto] {$\scriptstyle \Omega_1$} (u);
          \draw[|->] (l.east) to[bend left=0] node[pos=.5,auto] {$\scriptstyle \Omega_2$} (r);
          \draw[|->] (l.south east) to[bend right=6] node[pos=.5,auto,swap] {$\scriptstyle \Omega_{n-2}$} (d);
        \end{tikzpicture}
      \end{equation*}
      all continuous: $\Omega_{2}$ and $\Omega_{n-2}$ as in the proof of Lemma \ref{le:def.latt.mor} (the portion invoking \cite[Proposition 13.4]{salt_divalg}), and $\Omega_1$ by the same token, by the continuity of $\phi$ and that of $\Ad_SV$ (as a function of $\cP$), since
      \begin{equation*}
        \bC TSv=\ker (I_n+\Ad_{TS}V)=\ker(I_n+\phi(\Ad_SV)).
      \end{equation*}
      The loci
      \begin{equation*}
        \cP_{\bullet}
        :=
        \left\{p\in \cP\ :\ \Omega_1(p)\le \Omega_{\bullet}(p)\right\}
        ,\quad \bullet\in \{2,n-2\}
      \end{equation*}
      are closed and partition the \emph{connected} space $\cP$; because we observed above that $\cP_2$ is not empty, it must be all of $\cP$. Thus, for 
\begin{equation*}
   w:=S^{-1}v=\sum_{j=1}^n c_j \frac 1{\lambda_j}e_j\in (Sv^{\perp})^{\perp}
\end{equation*} 
we have $w \in \ker(I_n+U)$, so
      \begin{equation*}
       \sum_{j=1}^n c_j \lambda'_j e_j=TSv\in \ker(I_n+U) = \spn\left\{Sv,\ w\right\}=\spn\left\{\sum_{j=1}^n c_j \lambda_j e_j,\ \sum_{j=1}^n c_j \frac 1{\lambda_j}e_j\right\}.
     \end{equation*}
      In particular, as all $c_1, \ldots , c_n$ are nonzero, we conclude that the tuples $(\lambda_j)$, $(1/\lambda_j)$ and $(\lambda'_j)$ are linearly dependent. The same reasoning applies also to the orthogonal complement
      \begin{equation*}
        \bC\left(\sum_{j=1}^n c_j \frac 1{\lambda'_j}e_j\right)
        =
        (TSv^{\perp})^{\perp},
      \end{equation*}
      so that vector too must belong to the same $(\le 2)$-dimensional span:
      \begin{equation*}
        \dim \spn\left\{(\lambda_j),\ (1/\lambda_j),\ (\lambda'_j),\ (1/\lambda'_j)\right\}\le 2.
      \end{equation*}
    There are thus $a,b,c,d\in \bC$ with
    \begin{equation*}
      \forall 1\le j\le n
      \quad:\quad
      \lambda'_j = a\lambda_j+b\frac 1{\lambda_j}
      \quad\text{and}\quad
      \frac 1{\lambda'_j} = c\lambda_j + d\frac 1{\lambda_j}
    \end{equation*}
    This also reads
    \begin{equation}\label{eq:abcd}
      \exists a,b,c,d\in \bC
      \quad:\quad
      \left(a\lambda_j+b\frac 1{\lambda_j}\right)
      \left(c\lambda_j + d\frac 1{\lambda_j}\right)
      =1    
      \quad\text{and hence}\quad
      \left(a\lambda_j^2+b\right)
      \left(c\lambda_j^2 + d\right)
      =\lambda_j^2
    \end{equation}
    for all $1 \leq j \leq n$.  Because $n\ge 3$, for a generic choice of $\lambda_j^2$ the only $(a,b,c,d)$ witnessing \eqref{eq:abcd} are those that produce a tautological equation (i.e.\ $\lambda_j^2=\lambda_j^2$). In other words, for a dense set of $(\lambda_j)$ we can take
    \begin{equation*}
      (a,b,c,d)\in \{(1,0,0,1),\ (0,1,1,0)\} \xRightarrow{\quad} (\lambda_j') \in \{(\lambda_j), (1/\lambda_j)\}
    \end{equation*}
    and hence by \eqref{eq:exists.ps} and \eqref{eq:definition of TSv} we obtain
 \begin{equation*}
     T=T_S=T_{(\lambda_j)}\in \{I_n,\ S^{-2}\}.
 \end{equation*}

    Continuity ensures that this holds for all $S\ge 0$, finishing the proof.
    
  \item \textbf{: The same $p=p_S\in \{0,-2\}$ functions for all $S$.} As already observed, the two subsets \eqref{eq:twops} are closed. Their intersection is $\bR_{>0} I_n$, for the sets $\Ad_S \U(n)$ and $\Ad_{S^{-1}}\U(n)$ are distinct for every non-scalar positive invertible $S$. The intersections of \eqref{eq:twops} with the connected $\GL(n)_{\ge 0}\setminus \bR_{>0}I_n$ partition that space, so one of the two must constitute all of $\GL(n)_{\ge 0}$.  \qedhere
  \end{enumerate}
\end{proof}

\begin{proof}[Proof of Proposition~\ref{pr:still.theta}]
  Note that $\phi$ is spectrum-preserving by Corollary \ref{cor:sp.shrk.conn.conf.sp}. 

  \begin{enumerate}[(I),wide]

  \item\label{item:pr:still.theta:gl}  \textbf{: $\ca{X}_n = \mathrm{GL}(n)_{ss}$.}  We already know from the unitary branch of Theorem \ref{th:main-result-gln} that $\phi|_{\U(n)}$ is of the form \eqref{eq:inner}, so we can recast the claim as
    \begin{equation*}
      \phi|_{\U(n)}=\id_{\U(n)}
      \xRightarrow{\quad}
      \phi\in\{\id_{\GL(n)_{ss}}, \Theta\}.
    \end{equation*}
    The result also applies to arbitrary maximal compact subgroups $\Ad_{S}\U(n)\le \GL(n)$, $S \in \GL(n)$, so that $\phi$ restricted to each such subgroup is of the form $\Ad_{T_S}\circ \bullet$, with $\bullet\in \{\id,\ (\cdot)^t\}$ and $T_S \in \GL(n)$. Furthermore, because we are assuming $\bullet=\id$ at least when $S=\id$, Lemma \ref{le:all.no.transp} shows that this is in fact always so (i.e.\ no transposition):
    \begin{equation}\label{eq:stu}
      \left(\forall S\in \GL(n)\right)
      \left(\exists T=T_S\in \GL(n)\right)
      \left(\forall U\in \U(n)\right)
      \quad:\quad
      \phi(SUS^{-1})
      =
      TSUS^{-1}T^{-1}
    \end{equation}
    There is no loss of generality in assuming the operators $S$ positive, once more leveraging polar decompositions. Proposition  \ref{pr:is.or.s2} then implies that
    \begin{equation*}
      \forall S\in \GL(n)_{\ge 0}
      \quad :\quad
      \phi|_{\Ad_S \U(n)}= \Ad_{S} \quad \text{ or } \quad \phi|_{\Ad_S \U(n)}=\Ad_{S^{-2}}\quad\left(\text{exclusively}\right).
    \end{equation*}
    In either case  the conclusion follows as in the normal-matrix branch of Theorem \ref{th:main-result-gln}: every diagonalizable matrix is a linear combination of mutually-commuting elements of $\Ad_S\U(n)$ for some $S\in \GL(n)$.

  \item\label{item:pr:still.theta:sl} \textbf{: $\ca{X}_n = \mathrm{SL}(n)_{ss}$.} We can proceed precisely as in step \ref{item:th:main-result-gln:sl} of the proof of Theorem \ref{th:main-result-gln}. By Proposition \ref{pr:alg.cpct.tor} $\phi$ restricts to all maximal tori as conjugation, and every semisimple element belongs to some maximal torus. It follows that $\phi\left(\SL(n)_{ss}\right)\subseteq \SL(n)_{ss}$, and following the same notation as in the proof of Theorem  \ref{th:main-result-gln}\ref{item:th:main-result-gln:sl} the map $\zeta$ of \eqref{eq:psi.conj.root.det} restricts in the present context to the continuous, spectrum- and commutativity-preserving extension of $\phi$ to $\GL(n)_{*ss} \supseteq \SL(n)_{ss}$  to which step \ref{item:pr:still.theta:gl} applies. The conclusion follows from the fact that all components of \eqref{eq:comp3maps} preserve scalar multiplication.  \qedhere
  \end{enumerate}
\end{proof}


To address the continuity of the assignment $\Theta$ from \eqref{eq:sns}, we employ the following extension \cite[Abstract]{zbMATH06285212} of of the usual \cite[\S I.6.2.4]{blk} normal-operator \emph{continuous functional calculus} to all of $(M_{n})_{ss}$.

\begin{definition}\label{def:fn.calc.ss}
  Write an arbitrary $T \in (M_{n})_{ss}$ as
  \begin{equation}\label{eq:tel}
    T=\sum_{\lambda\in \spc(T)} \lambda E_{\lambda},
  \end{equation} 
  where for each $\lambda \in \spc(T)$ the (not necessarily normal) idempotent $E_{\lambda}$ is uniquely determined by
  \begin{equation*}
    E_{\lambda}^2=E_{\lambda}
    ,\quad
    \im E_{\lambda}=\ker(\lambda I_n-T)
    ,\quad
    E_{\lambda}T = TE_{\lambda}.
  \end{equation*}
  We write
  \begin{equation}\label{eq:cont.fc.ss}
    f\left(T\right)
    :=
    \sum_{\lambda\in \spc(T)} f(\lambda)E_{\lambda}
    ,\quad
    \forall
    \left(\spc(T)\subset U=\overset{\circ}{U}\right)
    \xrightarrow[\quad\text{function}\quad]{f}
    \bC.
  \end{equation}
  The construction \eqref{eq:cont.fc.ss} is Ad-invariant, in the sense that $\Ad_S f(T)=f\left(\Ad_S T\right)$.  \hfill$\lozenge$  
\end{definition}

Note that \eqref{eq:cont.fc.ss} certainly is not \emph{continuous}, generally, on all of  $(M_{n})_{ss}$, even when $f$ is. Example~\ref{ex:fn.calc.disc} below discusses the case $n=2$. For $n\ge 3$ \eqref{eq:cont.fc.ss} is continuous on 
\begin{equation*}
  \left\{T\in (M_{n})_{ss}\ :\ \spc(T)\subseteq U \right\}
  \subseteq
  (M_{n})_{ss}
\end{equation*}
precisely \cite[Theorem 1.2(A)]{zbMATH06285212} when $f$ is holomorphic on $U$. 

\begin{example}\label{ex:fn.calc.disc}
  The matrix
  \begin{equation*}
    T:=
    \begin{pmatrix}
      \lambda_1&\alpha\\
      0&\lambda_2
    \end{pmatrix}
    \in
    M_2
  \end{equation*}
  is semisimple whenever $\lambda_i$ are distinct, with eigenvalues/eigenvectors
  \begin{equation*}
    \lambda_1\quad:\quad
    \begin{pmatrix}
      1\\0
    \end{pmatrix}
    \quad\text{and}\quad
    \lambda_2\quad:\quad
    \begin{pmatrix}
      \frac{\alpha}{\lambda_2-\lambda_1}\\1
    \end{pmatrix}.
  \end{equation*}
  For continuous function $f$ the matrix $f(T)$ will have eigenvalues $f(\lambda_j)$ ($j=1,2$) along the same lines, so must be the matrix
  \begin{equation*}
    f(T)
    =
    \begin{pmatrix}
      f(\lambda_1)&\frac{\alpha(f(\lambda_2)-f(\lambda_1))}{\lambda_2-\lambda_1}\\
      0&f(\lambda_2)
    \end{pmatrix}.
  \end{equation*}
  One can easily arrange for $f$ to be continuous in a neighborhood of $0=f(0)\in \bC$, with
  \begin{equation*}
    \alpha,\ \lambda_j
    \xrightarrow{\quad}
    0
    \quad\text{but}\quad
    \frac{\alpha(f(\lambda_2)-f(\lambda_1))}{\lambda_2-\lambda_1}
    \longarrownot
    \xrightarrow{\quad}
    0;
  \end{equation*}
  for those $f$, the family of semisimple matrices $f(\bullet)$ will be discontinuous at $0_2\in M_2$. \hfill$\lozenge$
\end{example}

Example~\ref{ex:fn.calc.disc} can be expanded into a characterization of the points of continuity for \eqref{eq:cont.fc.ss} (note the parallels to Proposition~\ref{pr:nbhd.eig.sel.m}). 

\begin{proposition}\label{pr:when.fn.calc.cont.op}
  The following conditions on a matrix $T\in (M_{n})_{ss}$ are equivalent:
  \begin{enumerate}[(a),wide]

  \item\label{item:pr:when.fn.calc.cont:cont} $f(\bullet)$ defined by \eqref{eq:cont.fc.ss} is continuous at $T$ for every function $f$ continuous in a neighborhood of $\spc(T)$.
    
  \item\label{item:pr:when.fn.calc.cont:sspec} $T$ has simple spectrum.
  
  \end{enumerate}
\end{proposition}
\begin{proof}
  \begin{enumerate}[label={},wide]
  \item \textbf{\ref{item:pr:when.fn.calc.cont:cont} $\Rightarrow$ \ref{item:pr:when.fn.calc.cont:sspec}.} If $T$ has at least one eigenvalue $\lambda$ of multiplicity $\ge 2$, we can restrict attention to a 2-dimensional subspace $W\le \ker(\lambda I_n-T)$ by working throughout with matrices leaving invariant a decomposition $\bC^n=W\oplus W'$. Example~\ref{ex:fn.calc.disc} then applies to produce (even globally) continuous functions $f$ with \eqref{eq:cont.fc.ss} discontinuous at $T|_W=\lambda \id_W$.    
    
  \item \textbf{\ref{item:pr:when.fn.calc.cont:sspec} $\Rightarrow$ \ref{item:pr:when.fn.calc.cont:cont}.} The simplicity of the spectrum means (by spectrum continuity \cite[Theorem 3]{Newburgh}) that for any convergent sequence $T_k\xrightarrow[k]{}T$ and any fixed eigenvalue $\lambda\in \spc(T)$ we have
    \begin{equation*}
      \spc(T_k)
      \ni
      \lambda_k
      \left(\text{simple for large $k$}\right)
      \xrightarrow[\quad k\quad]{\quad}
      \lambda.
    \end{equation*}
   
      The one-dimensional kernels
      \begin{equation*}
        \ker\left(T_k-\lambda_k I_n\right)
        =
        \im E_{\lambda_k}
      \end{equation*}
      converge (in $\bG$) to the $\lambda$-eigenspace $\im E_{\lambda}$ of $T$ by yet another application of \cite[Proposition 13.4]{salt_divalg}. This goes for \emph{all} (simple!) $\lambda\in \spc(T)$, so we also have
      \begin{equation*}
         \ker E_{\lambda_k} =
        \bigoplus_{\lambda\ne \lambda'\in \spc(T)} \im E_{\lambda'_k}
       \xrightarrow[\quad k\quad]{\quad\text{convergence in $\bG$}\quad}
        \bigoplus_{\lambda\ne \lambda'\in \spc(T)} \im E_{\lambda'}
        =
        \ker E_{\lambda}.
      \end{equation*}
      This yields  
      \begin{equation}\label{eq:proj.conv}
        \forall \lambda\in\spc(T)
        \quad:\quad
        E_{\lambda_k}\xrightarrow[\quad k\quad]{\quad}
        E_\lambda,
      \end{equation}
    so
    \begin{equation*}
      f(T_k)=\sum_{\lambda_k\in \spc(T_k)}f(\lambda_k) E_{\lambda_k}
      \xrightarrow[\quad k\quad]{\quad}
      \sum_{\lambda\in \spc(T)}f(\lambda) E_{\lambda}
      =
      f(T)
    \end{equation*}
    by \eqref{eq:proj.conv} and the continuity of $f$.  \qedhere
  \end{enumerate}  
\end{proof}

The following result will help eliminate the extraneous factor in Proposition~\ref{pr:still.theta}. 

\begin{proposition}\label{pr:no.theta}
  Let $\ca{X}_n\in \{\GL(n)_{ss}, \SL(n)_{ss}\}$. The assignment $\Theta$ of \eqref{eq:sns} is a well-defined involutory spectrum- and commutativity-preserving map on $\cX_n$ for all $n\ge 1$, discontinuous for $n\ge 3$.
\end{proposition}
\begin{proof}
  First observe that if well-defined, $\Theta$ will be defined globally on $\ca{X}_n$: the semisimple matrices are precisely those conjugate to diagonal ones by some $S\in \GL(n)$, and the \emph{polar decomposition} \cite[Theorem 7.3.1]{hj_mtrx}
  \begin{equation*}
    S=|S^*|V
    ,\quad
    V\in \U(n)
    ,\quad
    |S^*|=(SS^*)^{\frac{1}{2}}\in \GL(n)_{\ge 0}
  \end{equation*}
  renders such a matrix an $|S^*|$-conjugate of a normal one.

  As to $\Theta$ being well-defined: for $S,T\in \GL(n)_{\ge 0}$ and $M,N\in N_n\cap \ca{X}_n$ we have 
  \begin{equation*}
    \begin{aligned}
      SNS^{-1} = TMT^{-1}
      &\xRightarrow{\quad}
        SN^*S^{-1} = TM^*T^{-1}
        \quad\left(\text{\emph{Putnam-Fuglede theorem} \cite[Problem 192]{hal_hspb_2e_1982}}\right)\\
      &\xRightarrow{\quad\text{taking conjugates}\quad}
        S^{-1}NS=T^{-1}MT.
    \end{aligned}
  \end{equation*}
  Clearly, the assignment $\Theta$ is spectrum-preserving. Furthermore, one of the proofs of Putnam-Fuglede (e.g.\ \cite[\S 41, Theorem 2]{halm_sm_2e_1957}) can also be adapted to prove that $\Theta$ is also commutativity-preserving: if $SNS^{-1}$ and $TMT^{-1}$ commute then so do $p(SNS^{-1})$ and $p(TMT^{-1})$ for polynomials $p$; apply this to a polynomial
  \begin{equation*}
    p\in \bC[x]
    \quad\text{with}\quad
    p(\lambda)=\overline{\lambda}
    ,\quad\forall \lambda\in \spc(M)\cup \spc(N),
  \end{equation*}
  whence $SN^*S^{-1}=p(SNS^{-1})$ and $TM^*T^{-1}=p(TMT^{-1})$ commute. Taking adjoints then proves that $S^{-1}NS$ and $T^{-1}MT$ do as well. 

  To assess continuity, observe that for normal $N$ and $S\in \GL(n)_{\ge 0}$ we have
  \begin{equation}\label{eq:adsfads}
    \Theta\left(\Ad_S N\right)
    =
    \Ad_{S^{-1}}N
    =
    \left(\Ad_S N^*\right)^*
    =
    \left(\Ad_S f(N)\right)^*
    =
    f\left(\Ad_S N\right)^*
  \end{equation}
  with $f(\bullet)$ as in \eqref{eq:cont.fc.ss} for $f:=\overline{\bullet}$ (complex conjugation). The continuity of $\Theta$ thus reduces to that of $f(\bullet)$:
    \begin{equation*}
      \text{$\Theta$ is continuous at }\Ad_SN
      \xLeftrightarrow[\quad\quad]{\quad\text{\eqref{eq:adsfads}}\quad}
      \text{$f(\bullet)$ is continuous at }\Ad_SN.
    \end{equation*}
  The latter condition fails for $n\ge 3$ by the aforementioned \cite[Theorem 1.2(A)]{zbMATH06285212}:
  \begin{itemize}[wide]
  \item The result applies directly to $\GL(n)_{ss}$, where continuity fails because complex conjugation is not holomorphic on the punctured complex plane $\bC^{\times}$.

  \item While for $\SL(n)$ the same conclusion follows from the $\GL(n)$ counterpart and the conjugate linearity of \eqref{eq:cont.fc.ss} for $f=\overline{\bullet}$:
    \begin{equation*}
      f(\alpha T) = \overline{\alpha}f(T)
      ,\quad
      \forall T\in \GL(n)_{ss}
      ,\quad
      \forall \alpha\in \bC.
    \end{equation*}
  \end{itemize}
\end{proof}

\begin{proof}[Proof of Theorem~\ref{th:ss}]
  That the maps \eqref{eq:inner} meet the requirements of course needs no proof, while Propositions \ref{pr:still.theta} and \ref{pr:no.theta} give the converse: the latter's discontinuity claim eliminates the $\Theta$ factor in \eqref{eq:comp3maps}.
\end{proof}

\def\polhk#1{\setbox0=\hbox{#1}{\ooalign{\hidewidth
  \lower1.5ex\hbox{`}\hidewidth\crcr\unhbox0}}}


\begin{thebibliography}{10}

\bibitem{Aupetit}
Bernard Aupetit.
\newblock Spectrum-preserving linear mappings between {B}anach algebras or
  {J}ordan-{B}anach algebras.
\newblock {\em J. London Math. Soc. (2)}, 62(3):917--924, 2000.

\bibitem{blk}
B.~Blackadar.
\newblock {\em Operator algebras}, volume 122 of {\em Encyclopaedia of
  Mathematical Sciences}.
\newblock Springer-Verlag, Berlin, 2006.
\newblock Theory of $C^*$-algebras and von Neumann algebras, Operator Algebras
  and Non-commutative Geometry, III.

\bibitem{BresarSemrl}
Matej Bre\v{s}ar and Peter \u{S}emrl.
\newblock An extension of the {G}leason-{K}ahane-\.zelazko theorem: a possible
  approach to {K}aplansky's problem.
\newblock {\em Expo. Math.}, 26(3):269--277, 2008.

\bibitem{btd_lie_1985}
Theodor Br{\"o}cker and Tammo tom Dieck.
\newblock {\em Representations of compact {Lie} groups}, volume~98 of {\em
  Grad. Texts Math.}
\newblock Springer, Cham, 1985.

\bibitem{bronson_matr}
Richard Bronson.
\newblock {\em Matrix methods. {An} introduction}.
\newblock Boston, MA: Academic Press, 2nd edition, 1991.

\bibitem{bbi}
Dmitri Burago, Yuri Burago, and Sergei Ivanov.
\newblock {\em A course in metric geometry}, volume~33 of {\em Graduate Studies
  in Mathematics}.
\newblock American Mathematical Society, Providence, RI, 2001.

\bibitem{zbMATH06468787}
Daryl Cooper and Jason~Fox Manning.
\newblock Non-faithful representations of surface groups into
  {{\(\mathrm{SL}(2,\mathbb C)\)}} which kill no simple closed curve.
\newblock {\em Geom. Dedicata}, 177:165--187, 2015.

\bibitem{zbMATH05719630}
Cl{\'e}ment de~Seguins~Pazzis.
\newblock The singular linear preservers of non-singular matrices.
\newblock {\em Linear Algebra Appl.}, 433(2):483--490, 2010.

\bibitem{zbMATH01747827}
Claude-Alain Faure.
\newblock An elementary proof of the fundamental theorem of projective
  geometry.
\newblock {\em Geom. Dedicata}, 90:145--151, 2002.

\bibitem{zbMATH05785888}
Robert Ghrist.
\newblock Configuration spaces, braids, and robotics.
\newblock In {\em Braids. Introductory lectures on braids, configurations and
  their applications. Based on the program ``Braids'', IMS, Singapore, May
  14--July 13, 2007.}, pages 263--304. Hackensack, NJ: World Scientific, 2010.

\bibitem{GogicPetekTomasevic}
Ilja Gogi\'c, Tatjana Petek, and Mateo Toma\v{s}evi\'c.
\newblock Characterizing {J}ordan embeddings between block upper-triangular
  subalgebras via preserving properties.
\newblock {\em Linear Algebra Appl.}, 704:192--217, 2025.

\bibitem{GOGIC2025129497}
Ilja Gogić and Mateo Tomašević.
\newblock An extension of {P}etek-\v{S}emrl preserver theorems for {J}ordan
  embeddings of structural matrix algebras.
\newblock {\em Journal of Mathematical Analysis and Applications},
  549(1):129497, 2025.

\bibitem{halm_sm_2e_1957}
P.~R. Halmos.
\newblock Introduction to {Hilbert} space and the theory of spectral
  multiplicity. 2nd ed.
\newblock New {York}: {Chelsea} {Publishing} {Company} 120 p. (1957)., 1957.

\bibitem{hal_hspb_2e_1982}
Paul~R. Halmos.
\newblock {\em A {Hilbert} space problem book. 2nd ed., rev. and enl},
  volume~19 of {\em Grad. Texts Math.}
\newblock Springer, Cham, 1982.

\bibitem{Herstein}
I.~N. Herstein.
\newblock Jordan homomorphisms.
\newblock {\em Trans. Amer. Math. Soc.}, 81:331--341, 1956.

\bibitem{hj_mtrx}
Roger~A. Horn and Charles~R. Johnson.
\newblock {\em Matrix analysis.}
\newblock Cambridge: Cambridge University Press, 2nd ed. edition, 2013.

\bibitem{jac_jord}
Nathan Jacobson.
\newblock {\em Structure and representations of {Jordan} algebras}, volume~39
  of {\em Colloq. Publ., Am. Math. Soc.}
\newblock American Mathematical Society (AMS), Providence, RI, 1968.

\bibitem{MR832991}
Ali~A. Jafarian and A.~R. Sourour.
\newblock Spectrum-preserving linear maps.
\newblock {\em J. Funct. Anal.}, 66(2):255--261, 1986.

\bibitem{Kaplansky}
Irving Kaplansky.
\newblock {\em Algebraic and analytic aspects of operator algebras}, volume No.
  1 of {\em Conference Board of the Mathematical Sciences Regional Conference
  Series in Mathematics}.
\newblock American Mathematical Society, Providence, RI, 1970.

\bibitem{lee_top-mfld_2e_2011}
John~M. Lee.
\newblock {\em Introduction to topological manifolds}, volume 202 of {\em Grad.
  Texts Math.}
\newblock New York, NY: Springer, 2nd ed. edition, 2011.

\bibitem{lee_smth-mfld_2e_2013}
John~M. Lee.
\newblock {\em Introduction to smooth manifolds}, volume 218 of {\em Grad.
  Texts Math.}
\newblock New York, NY: Springer, 2nd revised ed edition, 2013.

\bibitem{LiTsaiWangWong}
Chi-Kwong Li, Ming-Cheng Tsai, Ya-Shu Wang, and Ngai-Ching Wong.
\newblock Linear maps preserving matrices annihilated by a fixed polynomial.
\newblock {\em Linear Algebra Appl.}, 674:46--67, 2023.

\bibitem{LiZhang}
Chi-Kwong Li and Fuzhen Zhang.
\newblock Eigenvalue continuity and {G}er\v{s}gorin's theorem.
\newblock {\em Electron. J. Linear Algebra}, 35:619--625, 2019.

\bibitem{MR210096}
H.~R. Morton.
\newblock Symmetric products of the circle.
\newblock {\em Proc. Cambridge Philos. Soc.}, 63:349--352, 1967.

\bibitem{Newburgh}
J.~D. Newburgh.
\newblock The variation of spectra.
\newblock {\em Duke Math. J.}, 18:165--176, 1951.

\bibitem{zbMATH06285212}
Piotr Niemiec.
\newblock Functional calculus for diagonalizable matrices.
\newblock {\em Linear Multilinear Algebra}, 62(3):297--321, 2014.

\bibitem{pank_wign}
Mark Pankov.
\newblock {\em Wigner-type theorems for {Hilbert} {Grassmannians}}, volume 460
  of {\em Lond. Math. Soc. Lect. Note Ser.}
\newblock Cambridge: Cambridge University Press, 2020.

\bibitem{Petek-HM}
Tatjana Petek.
\newblock Mappings preserving spectrum and commutativity on {H}ermitian
  matrices.
\newblock {\em Linear Algebra Appl.}, 290(1-3):167--191, 1999.

\bibitem{Petek-TM}
Tatjana Petek.
\newblock Spectrum and commutativity preserving mappings on triangular
  matrices.
\newblock {\em Linear Algebra Appl.}, 357:107--122, 2002.

\bibitem{PetekSemrl}
Tatjana Petek and Peter \u{S}emrl.
\newblock Characterization of {J}ordan homomorphisms on {$M_n$} using
  preserving properties.
\newblock {\em Linear Algebra Appl.}, 269:33--46, 1998.

\bibitem{MR2736150}
Leiba Rodman and Peter \u{S}emrl.
\newblock A localization technique for linear preserver problems.
\newblock {\em Linear Algebra Appl.}, 433(11-12):2257--2268, 2010.

\bibitem{zbMATH03133061}
M.~Rosenblum.
\newblock On a theorem of {Fuglede} and {Putnam}.
\newblock {\em J. Lond. Math. Soc.}, 33:376--377, 1958.

\bibitem{salt_divalg}
David~J. Saltman.
\newblock {\em Lectures on division algebras}, volume~94 of {\em Reg. Conf.
  Ser. Math.}
\newblock Providence, RI: American Mathematical Society, 1999.

\bibitem{MR1311919}
A.~R. Sourour.
\newblock Invertibility preserving linear maps on {$\mathcal{L}(X)$}.
\newblock {\em Trans. Amer. Math. Soc.}, 348(1):13--30, 1996.

\bibitem{MR1866032}
Peter \u{S}emrl.
\newblock Invertibility preserving linear maps and algebraic reflexivity of
  elementary operators of length one.
\newblock {\em Proc. Amer. Math. Soc.}, 130(3):769--772, 2002.

\bibitem{Semrl2}
Peter \u{S}emrl.
\newblock Maps on matrix spaces.
\newblock {\em Linear Algebra Appl.}, 413(2-3):364--393, 2006.

\bibitem{Semrl}
Peter \u{S}emrl.
\newblock Characterizing {J}ordan automorphisms of matrix algebras through
  preserving properties.
\newblock {\em Oper. Matrices}, 2(1):125--136, 2008.

\end{thebibliography}
\end{document}